\documentclass[11pt]{article}

\usepackage[american]{babel}
\usepackage{libertine} 

\usepackage{amsmath,amssymb,amsthm,hyperref,color,dsfont}
\usepackage[margin = 1in]{geometry}
\usepackage{bbm}
\usepackage{bm}
\usepackage{verbatim}
\usepackage{graphicx}
\usepackage{color}
\usepackage{tikz}
\usepackage{subfigure}

\hypersetup{colorlinks=true,citecolor=blue, linkcolor=blue, urlcolor=blue}

\newtheorem{theorem}{Theorem}[section]

\newtheorem{lemma}[theorem]{Lemma}
\newtheorem{corollary}[theorem]{Corollary}
\newtheorem{remark}[theorem]{Remark}
\newtheorem{definition}[theorem]{Definition}

\numberwithin{equation}{section}
\numberwithin{figure}{section}

\renewcommand{\rho}{\varrho}
\renewcommand{\phi}{\varphi}

\newcommand{\R}{\mathbb{R}}

\newcommand{\IND}{{\bf 1}}
\newcommand{\scalar}[2]{\langle #1 , #2\rangle}
\newcommand{\tmix}{T_{\rm mix}}

\newcommand{\dual}{\mathrm{d}}

\DeclareMathOperator{\var}{Var}

\DeclareMathOperator{\dist}{dist}

\DeclareMathOperator{\Ent}{Ent}

\newcommand{\be}{\begin{equation}}

\newcommand{\cC}{\ensuremath{\mathcal C}} 
\newcommand{\cD}{\ensuremath{\mathcal D}} 
 
\newcommand{\cF}{\ensuremath{\mathcal F}}

\newcommand{\bbE}{{\ensuremath{\mathbb E}} }

\newcommand{\bbN}{{\ensuremath{\mathbb N}} }

\newcommand{\bbR}{{\ensuremath{\mathbb R}} }

\newcommand{\bbZ}{{\ensuremath{\mathbb Z}} }

\newcommand{\si}{\sigma} 
\newcommand{\ent}{{\rm Ent} } 

\newcommand{\wt}{\widetilde } 
\newcommand{\tc}{\mid}

\let\a=\alpha \let\b=\beta   \let\d=\delta

\let\r=\rho  \let\s=\sigma \let\t=\tau

\let\O=\Omega

\def\({\left(}
\def\){\right)}
\def\eric#1{\marginpar{$\leftarrow$\fbox{E}}\footnote{$\Rightarrow$~{\sf #1 --Eric}}}

\newcommand{\PSW}{P_{\textsc{sw}}}
\newcommand{\PSWI}{P^{(i)}_{\textsc{sw}}}
\newcommand{\PHB}{P_{\textsc{hb}}}
\newcommand{\PSB}{P_{\textsc{sb}}}
\newcommand{\PSWE}{{\wt P}_{\textsc{sw}}}
\newcommand{\Joint}{\O_{\textsc{j}}}
\newcommand{\EdgeSet}{\mathbb E}
\newcommand{\ignore}[1]{}
\newcommand{\integers}{{\mathbb Z}}

\title{Entropy decay in the Swendsen-Wang dynamics on $\mathbb{Z}^d$}

\author{
	Antonio Blanca\thanks{Department of Computer Science and Engineering, Penn State, University Park, PA 16801.
		Email: {ablanca@cse.psu.edu}. 
		Research supported in part by NSF grant CCF-1850443.}
	\and
	Pietro Caputo\thanks{Department of Mathematics, University of Roma Tre, Largo San Murialdo 1, 00146 Roma, Italy. Email: {pietro.caputo@uniroma3.it}}
	\and
	Daniel Parisi\thanks{Department of Mathematics, University of Roma Tre, Largo San Murialdo 1, 00146 Roma, Italy. Email: {daniel.parisi@uniroma3.it}}
	\and
	Alistair Sinclair\thanks{Computer Science Division, U.C. Berkeley, Berkeley, CA 94720. Email: {sinclair@cs.berkeley.edu}. Research supported in part by NSF grant CCF-1815328.}
	\and
	Eric Vigoda\thanks{School of Computer Science, Georgia Tech, Atlanta, GA 30332.
		Email: {vigoda@gatech.edu}. Research supported in part by NSF grant  CCF-2007022.}		
}

\begin{document}

\maketitle	

\begin{abstract}
We study the mixing time of the {\em Swendsen-Wang} dynamics
for the ferromagnetic Ising and Potts models on the integer lattice~${\mathbb Z}^d$.
This dynamics is a widely used Markov chain that has largely resisted sharp analysis
because it is {\em non-local}, i.e., it changes the entire configuration in one step.  
We prove that, whenever {\em strong spatial mixing (SSM)} holds, the mixing time on any
$n$-vertex cube in~${\mathbb Z}^d$ is $O(\log n)$, 
and we prove this is tight
by establishing a matching lower bound on the mixing time.
The previous best known bound was~$O(n)$.  SSM is a standard condition corresponding
to exponential decay of correlations with distance between spins on the lattice and
is known to hold in $d=2$ dimensions throughout the high-temperature (single phase)
region.  Our result follows from a {\em modified log-Sobolev inequality},
which expresses the fact that the dynamics contracts relative entropy at a constant
rate at each step.  The proof of this fact utilizes a new factorization of the entropy
in the joint probability space over spins and edges that underlies the Swendsen-Wang
dynamics, which extends to general bipartite graphs of bounded degree. 
This factorization leads to several additional results, including mixing
time bounds for a number of natural local and non-local Markov chains on the
joint space, as well as for the standard random-cluster dynamics.
\end{abstract}

\thispagestyle{empty}

\newpage

\setcounter{page}{1}

\section{Introduction}

The ferromagnetic Potts model is 
a classical spin system in statistical physics and theoretical computer science.
It is specified by a finite graph $G=(V,\EdgeSet)$, a set of {\it spins\/} (or colors)
$[q] = \{1,\dots,q\}$,
and an {\it edge weight\/} or inverse temperature parameter $\beta  > 0$.
A configuration $\sigma \in \Omega=\{1,\dots,q\}^V$ 
of the model assigns a spin value
to each vertex $v \in V$, and  
the probability of finding the
system in a given configuration~$\sigma$ is given by the {\it Gibbs\/} (or {\it Boltzmann\/})
distribution
\begin{equation}
\label{eq:Gibbs}
\mu(\sigma) = \mu_{G,\beta}(\sigma) := \frac{1}{Z} \exp(-\b |D(\s)|),
\end{equation}
where $D(\s) := \{\{v, w\} \in \EdgeSet : \s_v \neq \s_w\}$ is the set of edges whose endpoints
have disagreeing spins in $\s$
and $Z := \sum_{\sigma\in\Omega}  \exp(-\b |D(\s)|)$ is the normalizing factor or \emph{partition function}.
Note that this model is {\it ferromagnetic}, in the sense that neighboring spins want to align
with each other.  The Ising model of ferromagnetism is exactly the case $q=2$.

We focus on the classical setting
where $G$ is a subgraph of the infinite $d$-dimensional lattice~$\bbZ^d$.
We will mostly restrict attention to the case where $V=\{0,\dots,\ell\}^d$ is a cube, but
our results can be extended to more general subgraphs of~$\bbZ^d$; see Remark~\ref{rmk:regions}.
In fact, our main technical contributions apply to general bipartite graphs of bounded degree.

A popular Markov chain for sampling from the Gibbs distribution \eqref{eq:Gibbs} is the 
{\it Swendsen-Wang (SW)\/} dynamics~\cite{SW}, which utilizes the random-cluster representation
of the Potts model to derive a sophisticated {\it non-local\/} Markov chain in which every vertex
can update its spin in each step. From the current spin configuration $\s(t) \in \Omega$, the SW dynamics generates $\s(t+1) \in \Omega$ as follows:
\begin{enumerate}
	\item Let $M(\s(t)) = \EdgeSet \setminus D(\s(t))= \{\{v, w\} \in \EdgeSet : \s_v(t) = \s_w(t)\}$ be the set of monochromatic edges of $G$ in $\s(t)$.
	\item Independently for each edge $e \in M(\s(t))$, retain~$e$ with probability $1 - \exp(-\b)$
	and delete it otherwise, resulting in the subset $A(t) \subseteq M(\s(t))$.  (This is 
	equivalent to performing bond percolation with probability $1 - \exp(-\b)$ on the
	subgraph $(V,M(\s(t)))$.)
	\item For each connected component $\cC$ in the subgraph $(V,A(t))$, independently choose a spin $s_{\cC}$ uniformly at
	random from $[q]$ and assign $s_{\cC}$ to all vertices in $\cC$, yielding $\s(t+1) \in \Omega$. 
\end{enumerate}
The Swendsen-Wang dynamics is ergodic, and has~\eqref{eq:Gibbs} as its stationary distribution; 
see~\cite{ES} for a proof.

This non-local dynamics has the ability to flip large regions of spins in one step and was thus originally proposed as an alternative algorithm for overcoming the slow convergence at low temperatures of the Glauber dynamics, the standard local Markov chain that updates the spin of a single, randomly chosen vertex in each step. 
At high temperatures where the Glauber dynamics is quite fast, the SW dynamics provides a powerful alternative sampling algorithm since one can efficiently parallelize its global steps.

In this paper, we are interested in the speed of convergence of the SW dynamics to stationarity, 
and in particular its \emph{mixing time}. The mixing time captures the convergence rate in
total variation distance of a Markov chain from the worst possible starting configuration and 
is the most standard measure of the speed of convergence. 
Results proving tight bounds for the mixing time of the SW dynamics are rare, and are 
limited to very special classes of graphs, such as the complete graph and 
trees \cite{Huber,LNNP,GSVmf,BSmf}, 
or to very high temperatures~\cite{MOSc,NS}. 
Most bounds for the mixing time of the SW dynamics are derived by comparison with the Glauber dynamics~\cite{Ullrich1}, and are consequently often very far from sharp. 
We also know of multiple examples where the mixing time of the SW dynamics is exponential in the number of vertices of the graph; see, e.g.,~\cite{GJ,GSVmf,BSmf,GLP,BCFKTVV,BCT}.

There is a long line of work studying the connection between 
{\it spatial mixing\/} (i.e., decay of correlations) properties of Gibbs distributions
and the speed of convergence of Markov chains~(see, e.g., \cite{Holley,AH,Zeg,SZ,MOI,MOII,Cesi,DSVW,MS}). 
These results focus on local Markov chains, such as the Glauber dynamics, 
but there has also been some recent progress in understanding this connection for non-local Markov chains such as the SW dynamics~\cite{BCV,BCSV,CP}. 
In particular, it was established in~\cite{BCSV} that the {\it strong spatial mixing (SSM)\/} property
implies that the mixing time $\tmix(SW)$ of the SW dynamics is $O(n)$, where $n := |V|$ is the number of vertices.

SSM is a standard formalization of decay of correlations in spin systems
and, roughly speaking, expresses the fact that the correlation between spins at different
vertices decreases exponentially with the distance between them.
More precisely, given a pair of fixed configurations $\psi$ and $\psi_u$ on the boundary of $V$ such that $\psi$ and $\psi_u$ differ only in the spin of the vertex $u$, the
effect on the (conditional) marginal distribution at a set $B \subset V$ 
decays exponentially with the
distance between $B$ and the disagreement at $u$;
see Section~\ref{sec:backgnd} for a precise definition.
Our main algorithmic result in this paper is that the mixing time of the SW
dynamics is in fact $O(\log n)$ whenever SSM holds, and this is tight.

\begin{theorem}
	\label{thm:intro:sw}
	In an $n$-vertex cube of $\bbZ^d$,
	for all integer $q \ge 2$,  SSM implies that for all boundary conditions $\tmix(SW) = \Theta(\log n)$.
\end{theorem}
We recall that a {\it boundary condition\/}~$\tau$ for the Potts model is a fixed assignment of spins to 
the boundary of $V$; in the presence of a boundary condition, we consider the Gibbs distribution on~$V$ 
conditional on the assignment~$\tau$ on the boundary of $V$. The case where there is no boundary condition is known as the {\it free boundary\/} case and is also covered by our results.



In $\bbZ^2$, SSM is known to hold for all $q \ge 2$ and all $\beta < \beta_c(q)$,
where $\beta_c(q) = \ln (1+\sqrt{q})$ 
is the uniqueness threshold \cite{BDC,KA1,MOS}. 
Therefore, we obtain the following immediate corollary of Theorem~\ref{thm:intro:sw}.

\begin{corollary}
	\label{cor:sw:intro}
	In an $n$-vertex square region of $\bbZ^2$, for all $q \ge 2$, all $\beta <\beta_c(q)$ and all boundary conditions, we have 
	$\tmix(SW) = \Theta(\log n)$.
\end{corollary}
The best previous bound in the setting of Corollary~\ref{cor:sw:intro} was $\tmix(SW) = O(n)$ and follows from the results in~\cite{BCSV}.
Nam and Sly \cite{NS} recently proved an $O(\log{n})$ mixing time bound (as well as the cutoff phenomenon) for the periodic boundary condition for sufficiently high temperatures ($\beta \ll \beta_c(q)$), a 
stronger assumption than SSM. 
In higher dimensions $d \ge 3$, SSM is not known to hold up to the corresponding uniqueness threshold (it is only known for sufficiently small $\beta$; see~\cite{Mart}), 
but we expect the SW dynamics to be rapidly mixing throughout the high temperature regime for all $d \ge 3$.

The key to our improved mixing time analysis is a novel {\it factorization of entropy\/}
based on the joint probability space of spins and edges that underlies the SW dynamics.  This
factorization implies that the relative entropy decays at a constant rate, which in turn 
implies a tight bound on the mixing time via a modified log-Sobolev inequality.
In contrast, previous bounds for the SW dynamics~\cite{Ullrich1,BSz2,GuoJerrum,BCSV,BCV} 
have used the spectral gap,
which inherently loses a factor of~$O(n)$ 
when transferred to mixing time bounds
and cannot deliver a tight result.  We discuss our
new technique and its ramifications in the next subsection.

{\it A priori\/} the correct order of the mixing time of the SW dynamics is unclear.
In some settings, such as on the complete graph (the mean-field Potts model)
for all $\beta$ below the uniqueness threshold, the dynamics mixes in $\Theta(1)$
steps~\cite{LNNP, GSVmf}.
In this paper, to complement our main result of an $O(\log{n})$ upper bound, 
we also establish a lower bound of $\Omega(\log{n})$
for all boundary conditions whenever SSM holds. 
To prove
our lower bound, we follow the strategy introduced by 
Hayes and Sinclair~\cite{HS}, who   
proved a tight lower bound on the mixing time of the local Glauber dynamics. However, the synchronicity and non-locality of the updates in the SW dynamics presents a significant obstacle to the adaptation of their techniques and some new ideas are required. The main novel ingredient in our proof of the lower bound is an analysis of the speed of propagation of disagreements under a coupling of the steps of the SW dynamics, provided SSM holds.
To establish this we use a recent breakthrough result of Duminil-Copin, Raoufi, and Tassion~\cite{DRT}
establishing exponential decay of correlations (i.e., weak spatial mixing) in the entire high-temperature phase for the Potts model on $\bbZ^d$.
We believe that our analysis of the speed of disagreement propagation could be useful for
establishing lower bounds for the mixing time of SW dynamics in other settings.

Our methods also provide new results for the low-temperature regime $\beta > \beta_c(q)$ in $\bbZ^2$ for specific boundary conditions. 
We say that a boundary condition $\tau$ is {\it monochromatic\/} if $\tau$ fixes the spin of every boundary vertex to the same color. 
One of the most fundamental open problems in the study of
the Glauber dynamics for the Ising and Potts model concerns the mixing time
at low temperatures with a monochromatic boundary~\cite{MarTo,LMST}. 
We provide new bounds for the mixing time of the SW dynamics in this setting.

\begin{theorem}
	\label{cor:sw-lt:intro}
	In an $n$-vertex square region of $\bbZ^2$, for all $q \ge 2$ and all $\beta > \beta_c(q)$ we have	$\tmix(SW) = O(n \log n)$ for the free or monochromatic boundary condition.
\end{theorem}

The best previously known bound for the mixing time of the SW dynamics 
in an $n$-vertex square region of $\bbZ^2$
when $\beta > \beta_c(q)$ was $O(n^2 \log^2 n)$, which follows from the results
in~\cite{BSz2,Ullrich1}; see also~\cite{MarSW} for better (sub-linear) bounds
for the mixing time when $q=2$ and $\beta \gg \beta_c(q)$.
The bound in Theorem~\ref{cor:sw-lt:intro} is likely not tight, and
establishing that the SW dynamics mixes in $O(\log n)$ steps in $\bbZ^2$ throughout the low-temperature regime remains an important open problem.
Furthermore, our result for low temperature with a monochromatic boundary does not extend to higher dimensions $d \ge 3$, since it crucially uses the self-duality for the associated random-cluster distribution on $\bbZ^2$. For the Ising model with $d \ge 3$, 
the state of the art seems to be the results from~\cite{MarSW} for $\beta \gg \beta_c(q)$.

\subsection{A new analytic tool}

A standard tool for the analysis of Markov chains are comparison inequalities, which
relate the spectral gap of a chain of interest to that of some
simpler chain that has already been analyzed. 
This approach has proven particularly useful for analysis of the SW dynamics;
indeed, some of the currently best known upper bounds on its mixing time are obtained in this manner~\cite{Ullrich1,BSz2,GuoJerrum,BCSV,BCV}. 
As mentioned earlier, this approach is unable to yield tight bounds on the mixing time
of the SW dynamics since bounds obtained via the spectral gap inherently introduce
a factor $\log(1/\mu_*)=\Omega(n)$ on the mixing time, where $\mu_* = \min_\si \mu(\si)$.

A potentially more powerful approach is to compare instead the (classical) log-Sobolev constants 
(see Definition~\ref{def:LSI}).  This yields mixing time bounds with only a 
$\log\log(1/\mu_*)=O(\log n)$ dependence on~$\mu_*$, which is potentially tight.
Unfortunately, however, log-Sobolev inequalities are not tight for the SW dynamics,
and the best possible mixing time bound obtained in this way would be~$O(n)$ 
(see Remark~\ref{rmk:log-sob} for details).

A {\em modified\/} log-Sobolev inequality (which essentially bounds the rate of decay of
relative entropy; see again Definition~\ref{def:LSI}) is a strictly weaker (and hence easier to satisfy)
inequality than the classical log-Sobolev inequality, but still strong enough to establish mixing time
bounds with the same dependence on~$\mu_*$.  
There have been several notable recent results
bounding the modified log-Sobolev constant for various Markov chains~\cite{CGM,HermonSalez}.
However, there are no prior results addressing the modified log-Sobolev constant
for the SW dynamics and, more generally, no comparison inequalities are available
for the modified log-Sobolev constant.
In this paper, we develop new machinery that essentially allows us to compare the
modified log-Sobolev contant of the SW dynamics with that of a much simpler dynamics,
and hence obtain tight mixing time bounds.  This comparison is at the level
of {\it entropy factorization\/} rather than the modified log-Sobolev constant itself,
as we now describe.

{\it Approximate factorization\/} of the entropy of the Gibbs distribution~$\mu$ 
with constant $C$
says that, for any nonnegative function $f:\Omega \mapsto \bbR_+$,
\begin{equation}
\label{ineq:approx-factorization}
\ent_\mu(f) \leq C \sum_{v\in V} \mu[\ent_v(f)],
\end{equation}
where 
$\mu[f] = \sum_{\sigma \in \Omega} \mu(\sigma)f(\sigma)$
and
$\ent_\mu(f) := \mu[f \log (f/\mu[f])]$ are the expectation and entropy, respectively, of~$f$ with respect to $\mu$,
and $\ent_v$ is the entropy with respect to the conditional distribution at vertex~$v$ given the spins of its neighbors.
Note that necessarily $C\ge 1$, and $C=1$ when $\mu$ is a product measure.
(The term``constant'' here indicates that~\eqref{ineq:approx-factorization} holds for
fixed~$C$, independent of~$f$.  In most applications, $C$ will in fact be a
constant independent of the size of the underlying graph~$G$; we will write $C=O(1)$ to 
indicate this.)  
Approximate factorization with $C=O(1)$
played a central role in classical results proving that SSM implies $O(n\log{n})$
mixing time of the Glauber dynamics; see, e.g.,~\cite{Zeg,SZ,MOII,Cesi}.

Until recently it was unclear how to apply this approach to more general Markov chains.
However, in a very recent paper, Caputo and Parisi~\cite{CP} took an important step in
this direction by extending the above factorization as follows:
\begin{definition}
	\label{def:even-odd}
	For a spin configuration~$\si$ on a bipartite graph~$G$, let $\si_E$ (respectively, $\si_O$)
	denote the spins on the even (respectively, odd) side of the bipartition.
	We say that {\em approximate even/odd factorization}
	with constant $C$ holds if
	for all functions $f: \Omega \mapsto \bbR_+$,
	\begin{equation}\label{eq:eofactor}
	\ent_\mu \left(f\right) \leq C(\mu\left[\ent_\mu (f\tc \si_E)\right] + \mu\left[\ent_\mu (f\tc \si_O)\right]).
	\end{equation}
\end{definition}
\noindent
To clarify the meaning of the inequality \eqref{eq:eofactor}, we use the notation $\ent_\mu (f\tc \si_E)$ for the entropy of $f$ with respect to the conditional probability $\mu(\cdot\tc\si_E)$, that is the Gibbs measure conditioned on a given realization of the even spins $\si_E$, with similar notation for odd spins. In particular, taking the expectation one has
\[
\mu\left[\ent_\mu (f\tc \si_E)\right] =  \mu[f \log (f/\mu[f\tc\si_E])]\,,\quad \mu\left[\ent_\mu (f\tc \si_O)\right] =  \mu[f \log (f/\mu[f\tc\si_O])].
\]
 Caputo and Parisi~\cite{CP} showed that for spin systems on $\bbZ^d$, under the SSM assumption, approximate
even/odd factorization holds with $C=O(1)$, and used this fact to establish a tight mixing time bound
for  ``block dynamics'' in~$\mathbb{Z}^d$, a generalization of Glauber
dynamics in which a randomly chosen block of spins (rather than a single spin)
is updated in each step.

Our main analytic tool in this paper establishes that, on any bounded degree bipartite graph, 
even/odd factorization is in fact sufficient to 
ensure $O(\log n)$ mixing time for the much more complex SW dynamics.  
Note that SW is very far from a block dynamics, in that the configurations of
multiple, dynamically changing clusters of spins are updated simultaneously in each step.
\begin{theorem}
	\label{thm:evenodd-comparison}
	For all constant $\Delta$, for any bipartite graph of maximum degree $\Delta$, if the Gibbs distribution satisfies approximate even/odd factorization with $C=O(1)$ then 
	the mixing time of the Swendsen-Wang dynamics is $O(\log{n})$. 
	\end{theorem}
\noindent 
We remark that Theorem~\ref{thm:evenodd-comparison} 
holds for arbitrary boundary conditions (or pinnings of vertices) of the bipartite graph,
and thus Theorem~\ref{thm:intro:sw} for the lattice $\bbZ^d$ follows immediately from this theorem and the above mentioned results in~\cite{CP}. 
%

The main technical step in the proof of Theorem~\ref{thm:evenodd-comparison} is to show that
even/odd factorization implies a novel \emph{spin/edge factorization} of entropy (see Defintion \eqref{def:spin-edge} below), which is tailored to the SW dynamics so that it implies $O(\log n)$ mixing fairly directly.

\subsection{The spin/edge factorization}
\ignore{
	The SW dynamics is not a block dynamics as all vertices can be updated in a single step.
	Hence it is unclear to how to extend the notion of block factorization to capture the SW dynamics.
	We derive a novel factorization of entropy corresponding to the SW dynamics by
	utilizing the joint probability space on spins and edges
	introduced by Edwards and Sokal~\cite{ES}.}

Our new entropy factorization is based on the joint probability space on spins and edges
introduced by Edwards and Sokal~\cite{ES}, that underlies the SW dynamics.
Let $\Joint =\Omega \times \{0,1\}^\EdgeSet$ 
be the set of joint configurations $(\sigma,A)$ consisting of a spin assignment to the vertices $\sigma \in \Omega$ and a subset of edges $A \subseteq \EdgeSet$,
where recall that $\EdgeSet$ is the set of edges with both endpoints in $V$.	
The Edwards-Sokal distribution on $G$ with parameters $p\in[0,1]$ and $q\in\bbN$,
and free boundary condition, is the probability measure
on $\Joint$ given by 
\begin{equation}\label{jointnu-intro}
\nu(\si,A) := \frac1{Z_\textsc{j}}\,p^{|A|}(1-p)^{|\EdgeSet|-|A|}\IND(\si\sim A),
\end{equation}
where $\si\sim A$ means that 
$A \subseteq M(\si)$ 
(i.e., that every edge in $A$ is monochromatic in $\sigma$)
and $Z_\textsc{j}$ is the corresponding normalizing constant or partition function. 
When $p = 1 - e^{-\b}$, the ``spin marginal'' of $\nu$ is precisely the Potts distribution $\mu$ and $Z = Z_\textsc{j}$; the ``edge marginal'' of $\nu$ corresponds to the well-known random-cluster measure; see~\cite{FK,Grimmett}.  
The SW dynamics alternates between spin configurations and joint spin/edge configurations
in a manner consistent with~\eqref{jointnu-intro}.

We note that a boundary condition on the joint space allows fixing the state of \emph{both} spins
and edges and thus may introduce more complex dependencies. While 
our results in the joint space are stated here only for the free boundary condition, they actually extend to any \emph{spin-only} boundary condition. By a ``spin-only'' boundary condition we mean any boundary condition that fixes the spins of a subset of vertices, and fixes no values for the edges. 
In fact, in $\bbZ^d$, we can handle a slightly more general class of boundary conditions we call \emph{admissible} (see Definition~\ref{def:admissible}) which will be useful for proving~Theorem~\ref{cor:sw-lt:intro} and our results for random-cluster dynamics.

Our entropy factorization for the SW dynamics is defined as follows. 
\begin{definition}
	\label{def:spin-edge}
	We say that {\em approximate spin/edge factorization} with constant $C$ holds if 
	for all functions $f: \Joint \mapsto \bbR_+$,
	\begin{align}\label{entfact2o-intro}
	\ent_\nu (f)\leq C \,\left(\nu\left[\ent_\nu (f\tc\si)] +\nu[\ent_\nu (f\tc A)\right]\right).
	\end{align}
\end{definition}
Let us explain the terms in~\eqref{entfact2o-intro} in more detail. We write $\nu(\cdot\tc\si)$ for the probability obtained from $\nu$ by conditioning the on whole  spin configuration being equal to a given $\si\in\Omega$ and $\nu(\cdot\tc A)$ for the probability obtained from $\nu$ by conditioning on the whole edge configuration being equal to a given $A \subseteq \EdgeSet$. With this notation, $\ent_\nu (f\tc\si)$ and $\ent_\nu (f\tc A)$ denote 
the entropy of $f$ with respect to the conditional measures $\nu(\cdot\tc\si)$ and $\nu(\cdot\tc A)$, respectively. 
Therefore, taking their expectation with respect to $\nu$ one obtains
\begin{align}\label{condenti}
\nu\left[\ent_\nu (f\tc \si)\right] =  \nu[f \log (f/\nu[f\tc\si])]\,,\quad \nu\left[\ent_\nu (f\tc A)\right] =  \nu[f \log (f/\nu[f\tc A])].
\end{align}


The main technical ingredient in proving Theorem~\ref{thm:evenodd-comparison} is the
following ``comparison lemma'' for entropy factorization.
\begin{lemma}\label{lem:main-intro}
	For the Potts model at inverse temperature $\beta$ on any bipartite graph of maximum degree $\Delta$, approximate even/odd factorization 
	with constant $C$
	implies approximate spin/edge factorization with constant $C' = C'(C,\Delta,q,\beta)$.
\end{lemma}

To complete the proof of Theorem~\ref{thm:evenodd-comparison}, we show that the
spin/edge factorization in~\eqref{entfact2o-intro} implies decay of entropy for the
SW dynamics: namely,
there exists a constant $\delta > 0$ such that, for all functions $f:\Omega \mapsto \bbR_+$, we have
\begin{align}
\label{relentsw-pottso}
\ent_\mu( \PSW f) \leq (1-\d) \ent_\mu(f),
\end{align}
where $\PSW$ denotes the transition matrix of the SW dynamics.
As we recall in Section \ref{subseq:prel}, standard arguments then 
imply a modified log-Sobolev inequality, and a bound of $O(\log{n})$ on the mixing time of the SW dynamics.

\begin{lemma}\label{lem:main-intro-se-to-mixing}
	For the Potts model on \emph{any} $n$-vertex graph, approximate spin/edge factorization with constant $C=O(1)$ implies that \eqref{relentsw-pottso} holds with $\d=1/C$ and hence $\tmix(SW) = O(\log n)$. 
\end{lemma}

\subsection{Further results}
Our new entropy factorization framework
leads to several additional algorithmic results on $\bbZ^d$ that hold under the condition of SSM,
which we briefly summarize here.
First, we prove optimal $O(\log n)$ mixing time for {\it alternating systematic scan
	dynamics}, a natural non-local dynamics in which even and odd sides of the bipartition
are updated on alternate steps.  Systematic scan dynamics, in which updates are
performed in a deterministic rather than random sequence, are widely used in 
practice but are non-reversible and
typically much harder to analyze.  Second, we are able to show that various versions
of the SW dynamics on the {\it joint spin/edge space\/} mix in $O(\log n)$ time, as
does the SW dynamics for the {\it random-cluster model\/}. 
Finally, we show that a natural local Glauber dynamics in the joint space has optimal mixing time $\Theta(n\log n)$.
Formal statements of all these results can be found in the main body of the paper.

\medskip\noindent
{\bf Organization of the paper:} 
	In Section~\ref{sec:backgnd} we gather definitions of
various standard concepts used throughout the paper.  
Section~\ref{sec:entropysw} proves Lemma~\ref{lem:main-intro-se-to-mixing}
showing that the spin/edge entropy factorization implies $O(\log{n})$ mixing
for the SW dynamics.
We prove Lemma~\ref{lem:main-intro} relating 
even/odd factorization to spin/edge factorization in Section~\ref{sec:factorization}, and then
combine Lemmas~\ref{lem:main-intro} and~\ref{lem:main-intro-se-to-mixing}
to establish our main technical tool (Theorem~\ref{thm:evenodd-comparison}) and 
our main algorithmic result (Theorem~\ref{thm:intro:sw}).
Our lower bound on the mixing time is proved in Section~\ref{sec:lower-bound}.
We discuss further applications in the remaining sections.  
Section~\ref{sec:entropyjoint} proves
entropy decay for non-local and local dynamics in the joint space, and  
Section~\ref{sec:altscan} discusses the alternating scan dynamics.
Finally, we address the random-cluster dynamics in
Section~\ref{sec:rc}, concluding with a proof of Theorem~\ref{cor:sw-lt:intro}.

\section{Background}\label{sec:backgnd}
In this section, we formally define the spatial mixing property to be used throughout the paper. We also recall some known relations and prove some preliminary facts concerning entropy and mixing times.


\subsection{Strong spatial mixing (SSM)}

We assume $V \subset \bbZ^d$ is a $d$-dimensional cube of $\bbZ^d$. That is,
$
V=\{0,1,\dots,\ell\}^d
$
where $\ell$ is a positive integer.
We use $\partial V \subseteq V$ to denote the {\it internal boundary\/} of $V$; i.e., the set of vertices in $V$ adjacent to at least one vertex in $\bbZ^d \setminus V$.
A {\it boundary condition\/} $\psi$ for~$V$ is an
assignment of spins to some (or all) vertices in~$\partial V$; i.e., $\psi: U^{\psi} \rightarrow [q]$ with $U^\psi \subset \partial V$. 
The boundary condition where $U^\psi = \emptyset$ is called the \textit{free boundary condition}.
Given a boundary condition~$\psi$, each configuration $\sigma \in \Omega$
that agrees with $\psi$ on $U^{\psi} $ is assigned  probability
\[\mu^{\psi} (\sigma) = \frac{1}{Z^\psi} \cdot {e}^{-\beta |D(\sigma)|},\]
where $Z^\psi$ is the corresponding normalizing constant 
and $D(\s) := \{\{v, w\} \in \EdgeSet : \s_v \neq \s_w\}$.  We define $\mu^{\psi} (\sigma)=0$
for $\sigma\in\Omega$ that does not agree with~$\psi$.

Let $\mathcal{C}(V,a,b)$ be the property that, 
for all $B \subset V$, all $u \in \partial V$ and any pair of boundary conditions $\psi$, $\psi_u$ on $\partial V$ that differ only in the spin of the vertex $u$, we have
\begin{equation}
\label{prelim:eq:ssm}
{\|\mu_B^\psi \,-\, \mu_B^{\psi_u}\|}_{\textsc{tv}} \,\,\le\,\,  b\, \exp(-a \cdot \dist(u, B)),
\end{equation}	
where 
$\mu_B^\psi$ and $\mu_B^{\psi_u}$ are the probability measures induced in $B$ 
by the Potts distribution with boundary conditions
$\psi$ and ${\psi_{u}}$, respectively,
$\|\cdot\|_{\textsc{tv}}$ denotes total variation distance
and $\dist(u, B) = \min_{v \in B} {{\|u-v\|}_1}$.		
\begin{definition}	
	\label{def:prelim:ssm}
	We say that {\em strong spatial mixing (SSM)} holds if 
	there exist $a,b > 0$ such that
	$\mathcal{C}(V,a,b)$ holds for every
	cube $V \subset \bbZ^d$.		
\end{definition}

We note that 
the definition of SSM varies in the literature,
but we work here with one of the weakest (easiest to satisfy) versions.
In $\bbZ^2$, this form of SSM has been established for all $q \ge 2$ and $\beta < \beta_c(q)$, where $\beta_c(q) = \ln(1+\sqrt{q})$ is the uniqueness threshold~\cite{BDC,KA1,MOS}.  Finally, we stress that
the SSM property is determined only by the values of the parameters~$q$ and $p=1-e^{-\beta}$,
and not by any particular boundary condition.

\begin{remark}
	\label{rmk:regions}
	For definiteness, we have stated all of our results for $n$-vertex $d$-dimensional cubes but they extend to more general regions of $\bbZ^d$.
	In particular, we can consider regions which are the union 
	of disjoint translates of a given large enough cube.
	The variant of the SSM condition that requires   
	$\mathcal{C}(U,a,b)$ to hold for every such region $U$
	is equivalent to the one in Definition~\ref{def:prelim:ssm} (see~\cite[Theorem 2.6]{MOI}).
	As  noted in \cite{MOI}, a version of SSM which requires $\mathcal{C}(V,a,b)$ to hold for arbitrarily shaped regions $V$
	does not hold all the way to the uniqueness threshold.
\end{remark}

\subsection{Mixing time, entropy,  and log-Sobolev inequalities}\label{subseq:prel}


Let $P$ be the transition matrix of an ergodic Markov chain with finite state space $\Gamma$ and stationary distribution $\pi$.
Let $P^t(X_0,\cdot)$ denote the distribution of the chain after $t$ steps 
starting from the initial state $X_0 \in \Gamma$. The {\it mixing time} $\tmix(P)$ of the chain is defined as
\[
\tmix(P) = \max\limits_{X_0 \in \Gamma}\min \left\{ t \ge 0 : {\|{P}^t(X_0,\cdot)-\pi\|}_{\textsc{tv}} \le 1/4 \right\}.
\]
To prove upper bounds on the mixing time, in this paper we mostly rely on functional inequalities related to entropy.   

For a function $f:\Gamma \mapsto \bbR$, 
let $\pi[f] = \sum_{\s \in \Omega}  \pi(\s) f(\s)$ 
and $\var_\pi(f) = \pi[f^2] - \pi[f]^2$
denote its mean and variance with respect to $\pi$.
Likewise, for $f$ positive, the {\em entropy} of $f$ with respect to $\pi$ is defined as
\begin{align}\label{eq:def-ent}
\ent_\pi( f) = \pi\left[f \cdot \log \left(\frac{f}{\pi[f]}\right)\right] = \pi[f \cdot \log f] - \pi[f] \cdot \log \pi[f].
\end{align}

We often consider these functionals and their conditional versions with respect to the Potts measure $\mu$ or the joint measure $\nu$ (as defined in~\eqref{eq:Gibbs} and~\eqref{jointnu-intro} respectively). 
In particular, if the function $f$ is such that $f:\Joint \mapsto \bbR_+$,
for fixed $\si\in\Omega$ and $A \subseteq \EdgeSet$, we write
$
\nu[f \tc \si] = \sum_{A \subseteq \EdgeSet} \nu(A \tc \si) f(\sigma, A),
$
$
\nu[f\tc A] = \sum_{\sigma \in \Omega} \nu(\si \tc A) f(\sigma, A)
$
and
\begin{equation*}
\ent_\nu (f\tc\si) = \nu\left[f \cdot \log\left(\frac{f}{\nu[f\tc\si]}\right)\,\middle|\, \si\right],
\quad\quad
\ent_\nu (f\tc A) = \nu\left[f \cdot \log\left(\frac{f}{\nu[f\tc A]}\right) \,\middle|\, A\right].
\end{equation*}
Note that $\ent_\nu (f\tc\si)$ and $\ent_\nu (f\tc A)$ are functions of $\si\in\Omega$ and $A \subseteq \EdgeSet$, respectively, and 
with slight abuse 
of notation, we write $\nu\left[\ent_\nu (f\tc\si)\right]$ and $\nu\left[\ent_\nu (f\tc A)\right]$ for the corresponding expectations with respect to $\nu$; see \eqref{condenti}. The following identities hold:
\begin{align}
\Ent_\nu(f) &= \Ent_\nu(\nu[f\tc A]) + \nu[\Ent_\nu(f\tc A)]  \label{eq:total:entEdges};\\
\Ent_\nu(f) &= \Ent_\nu(\nu[f\tc\sigma]) + \nu[\Ent_\nu(f\tc \sigma)]. \label{eq:total:entSpins}
\end{align}
Indeed, both statements follow from the general decomposition 
\begin{align}
\Ent_\pi(f) = \Ent_\pi(\pi[f\tc\cF]) + \pi[\Ent_\pi(f\tc\cF)]  
\label{eq:total:ent-dec},
\end{align}
valid for any distribution $\pi$, and any sub $\si$-algebra $\cF$, which follows by adding and subtracting the term $\pi(f\log\pi[f\tc\cF])$ in \eqref{eq:def-ent}.
Another basic property of entropy that we shall use is the variational principle
\begin{equation}
\label{varprin}
\ent_\pi(f)=\sup \left\{\pi[f\varphi]\,,\; \pi[e^\varphi]\leq 1\right\}\,,
\end{equation}
valid for any distribution $\pi$, and any $f\geq 0$, where the supremum ranges over all functions $\varphi:\Gamma\mapsto \bbR$ such that $\pi[e^\varphi]\leq 1$, see e.g.\ Proposition 2.2 in \cite{Ledoux}. 

When $f\geq 0$ is such that  $\pi[f] =1$, then
$\ent_\pi (f) = H(f \pi\,|\, \pi)$ corresponds to the {\em relative entropy}, or Kullback-Leibler divergence, between the distribution $f \pi$ and $\pi$.
\begin{definition}\label{def:entdec}
	A Markov chain with transition matrix $P$ and stationary distribution $\pi$ is said to satisfy the
	{\em (discrete time) relative entropy decay with rate $\d > 0$} if for all distributions $\zeta$,
	\begin{align}\label{relent2}
	H(\zeta P\tc\pi)\leq (1-\d) H(\zeta \tc\pi).
	\end{align}
\end{definition}
We recall a well known consequence of entropy decay for the mixing time. 
For completeness, we include a proof.
\begin{lemma}\label{lem:mix}
	If a Markov chain with transition matrix $P$ and stationary distribution $\pi$ satisfies relative entropy decay with rate $\d>0$, then its mixing time $\tmix(P)$ satisfies
	\begin{align}\label{reltmix}
	\tmix(P) \leq 1+ \d^{-1}[\log(8)+\log\log(1/\pi_*)]\,, 
	\end{align}
	where $\pi_* = \min_\si \pi(\si)$. 
\end{lemma}
\begin{proof}[Proof of Lemma~\ref{lem:mix}]
	Pinsker's inequality says that 
	\begin{align}\label{pinsker}
	\|\d_\si P^n - \pi\|_{TV}^2\leq \frac12H(\d_\si P^n\tc\pi),
	\end{align}
	where $\d_\si (\tau) = \IND(\t = \s)$ is the Dirac mass at $\si$. 
	Iterating \eqref{relent2},
	\begin{align}\label{pinsker2}
	\|\d_\si P^n - \pi\|_{TV}^2\leq \frac12(1-\d)^nH(\d_\si\tc\pi).
	\end{align}
	Since $H(\d_\si\tc\pi)=-\log\pi(\si)$ and $(1-\d)^n\leq e^{-\d n}$ we obtain
	\begin{align}\label{pinsker3}
	\|\d_\si P^n - \pi\|_{TV}\leq \frac14,
	\end{align}
	as soon as $n$ is an integer such that $n\geq \d^{-1}\log[8\log(1/\pi_*)]$.
\end{proof}

\begin{remark}
	\label{rmk:entropy-mixing}
	If $\zeta$ has density $f$ with respect to $\pi$ (i.e., $\zeta = f\pi$), then $\zeta P$ has density $P^*f$ with respect to $\pi$, where $P^*$ is the {\em adjoint} or {\em time-reversal} matrix 
	$
	P^*(\si,\si') = \frac{\pi(\si')}{\pi(\si)}P(\si',\si).
	$
	Thus, \eqref{relent2} is equivalent to
	\begin{align}\label{relent3}
	\ent_\pi( P^* f) \leq (1-\d) \ent_\pi(f),
	\end{align}
	for all $f\geq 0$ such that $\pi[f]=1$. By homogeneity, this is equivalent to \eqref{relent3} for all $f\geq 0$. 
	When $P$ is reversible, that is when $P = P^*$,~\eqref{relent2} is equivalent to $\ent_\pi( P f) \leq (1-\d) \ent_\pi(f)$ for all $f\geq 0$.
\end{remark}
The inequality \eqref{relent3} can be considered as a discrete time analogue of the so-called modified log-Sobolev inequality characterizing the relative entropy decay for continuous time Markov chains; see, e.g.\ \cite{BT}.  Below we discuss some basic relations 
among \eqref{relent3}, the standard log-Sobolev inequality and the modified log-Sobolev inequality. 

Consider a transition matrix $P$ with stationary distribution $\pi$. The {\em Dirichlet form} associated to the pair $(P,\pi)$ is defined as 
\begin{align}\label{direlenti}
\cD_P(f,g)=
\scalar{f}{(1-P)g},
\end{align}
where $f,g$ are real functions on $\Gamma$, and  $\scalar{f}{g}=\pi[fg]$ denotes the scalar product in $L^2(\pi)$. 
Since $f$ is real we also have
\begin{align}\label{direlenti2}
\cD_P(f,f)=\scalar{(1-Q)f}{f}=\frac12\sum_{x,y}\pi(x)Q(x,y)(f(x)-f(y))^2,
\end{align}
where $Q=\tfrac12(P+P^*)$.  Moreover, if $P=P^*$ one has 
\begin{align}\label{direlenti21}
\cD_P(f,g)=\frac12\sum_{x,y}\pi(x)P(x,y)(f(x)-f(y))(g(x)-g(y)),
\end{align}
for all $f,g$.

\begin{definition}\label{def:LSI}
	The pair $(P,\pi)$ is said to satisfy the (standard) log-Sobolev inequality (LSI) with constant $\a$ if for all $f\geq 0$:
	\begin{align}\label{direlenti3}
	\cD_P(\sqrt f,\sqrt f)\geq \a\, \ent_\pi f.
	\end{align}
	It is said to satisfy the modified log-Sobolev inequality (MLSI) with constant $\d$ if for all $f\geq 0$:
	\begin{align}\label{direlenti4}
	\cD_P(f,\log f)\geq \d\, \ent_\pi f.
	\end{align}
\end{definition}
It is well known that the Log-Sobolev inequality is equivalent to the so-called hypercontractivity (see \cite[Theorem 3.5]{DSC}), 
while the modified Log-Sobolev inequality \eqref{direlenti4} is equivalent to exponential decay of the relative entropy with rate $\d$ for the continuous time kernel  $K_t=e^{(P-1)t}$  (see \cite[Theorem 3.6]{DSC}). Note that we are not assuming reversibility. 
To see the relation between the MLSI and the entropy decay in continuous time, note that if $K_t=e^{(P-1)t}$ and $f$ has mean $\pi[f]=1$ then using $K_t^*=e^{(P^*-1)t}$ one checks that the time derivative of the relative entropy satisfies
\begin{align}\label{relenti2}
\frac{d}{dt}\,H(\zeta K_t\tc\pi)&= \frac{d}{dt}\ent (K^*_tf)
= -\cD_P(K^*_tf,\log K_t^* f),
\end{align}
where $\zeta=f\cdot\pi$. 
Therefore \eqref{direlenti4} implies, for all $t\geq 0$:
$$
H(\zeta K_t\tc\pi)\leq H(\zeta\tc\pi)e^{-\d t}.
$$
Next, we observe that the bound \eqref{relent3} is stronger than the MLSI in \eqref{direlenti4}.
\begin{lemma}\label{lem:dtct}
	If  the entropy decay holds with rate $\d$ in discrete time then it holds with the same rate in continuous time.  That is, \eqref{relent3} implies the MLSI with constant $\d$.
\end{lemma}
\begin{proof}
	Suppose that 
	\begin{align*}
	\ent_\pi (P^* f)\leq (1-\d) \ent_\pi f.
	\end{align*}
	From the variational principle \eqref{varprin} it follows that for any $f\geq 0$ with $\pi[f]=1$:
	\begin{align*}
	\pi[(P^* f)\log f]\leq \ent P^* f.
	\end{align*}
	Therefore,
	\begin{align*}
	\cD_P(f,\log f)=\pi[((1-P^*) f)\log f]\geq \ent f -\ent P^* f\geq \d\,\ent f.
	\end{align*}
\end{proof}
It is well known that the standard LSI with constant $\a$ implies entropy decay in continuous time with rate $\d=2\a$, since $\cD_P( f,\log f)\geq2\cD_P(\sqrt f,\sqrt f)$ for all $f\geq 0$, and this can be improved to $\d=4\a$ in the reversible case; see \cite[Lemma 2.7]{DSC}. Here we recall a result of Miclo \cite{Miclo} showing in what sense the 
LSI implies the discrete time entropy decay. 
\begin{lemma}\label{lem:miclo}
	If the pair $(P^*P,\pi)$ satisfies the standard LSI with constant $\a$, then  the discrete time entropy decay holds for $(P,\pi)$ with constant $\d=\a$.
	In particular, if $P$ is reversible and $(P,\pi)$ satisfies the  LSI with constant $\a$, then for all $f\geq 0$:
	\begin{align}\label{arelentii3}
	\ent_\pi P f\leq (1-\a) \ent_\pi f.
	\end{align}
\end{lemma}
\begin{proof}
	The first assertion is proved in \cite[Proposition 6]{Miclo}. 
	The second assertion follows from the first and the simple observation that if $P=P^*$ then the LSI for $(P,\pi)$ implies the LSI for $(P^*P,\pi)$  with the same constant since $P^*P=P^2\leq P$ as quadratic forms in $L^2(\pi)$. 
\end{proof}

\section{Spin/edge factorization implies fast mixing: proof of Lemma~\ref{lem:main-intro-se-to-mixing}}\label{sec:entropysw}

As mentioned in the introduction, 
the proof of our main new analytic tool (Theorem~\ref{thm:evenodd-comparison})
has two components.
We show that approximate even/odd factorization implies spin/edge factorization (Lemma~\ref{lem:main-intro}), and then that spin/edge factorization implies 
$O(\log n)$ mixing for the SW dynamics (Lemma~\ref{lem:main-intro-se-to-mixing}).
In this section, we provide the proof of the latter result, whereas Lemma~\ref{lem:main-intro} is proved in the subsequent section.


\begin{proof}[Proof of Lemma~\ref{lem:main-intro-se-to-mixing}]
We show that the spin/edge factorization with constant $C$ implies that for all functions $f \ge 0$ with $\mu[f]=1$, one has
\begin{align}
\label{relentsw-main}
\ent_\mu( \PSW f ) \leq (1-\d) \ent_\mu(f),
\end{align}
with $\d=1/C$.
Since the SW dynamics is reversible with respect to $\mu$, we have $\PSW = \PSW^*$,
and the desired mixing time bound follows from Lemma~\ref{lem:mix} and Remark~\ref{rmk:entropy-mixing}. 

The transition matrix of the SW dynamics satisfies
$
\PSW(\si,\tau)= \sum_{A \subseteq M(\sigma)} \nu(A \tc \si) \nu(\tau \tc A),
$
where we recall that $M(\sigma)$ is the set of monochromatic edges in~$\sigma$.  Hence,
\begin{align*}
\PSW f (\si)&= \sum_{\tau \in \Omega} \PSW(\si,\tau)f(\tau) 
= \sum_{\tau \in \Omega} \sum_{A \subseteq M(\sigma)} \nu(A \tc \si) \nu(\tau \tc A) \hat f(\tau,A),
\end{align*}
where the function $\hat f: \Joint \mapsto \bbR_+$ is 
the ``lift'' of $f$ to the joint space, i.e., $\hat f (\sigma,A) = f(\sigma)$ for every $(\sigma,A) \in \Joint$. 
Recalling that  we write $\nu[f]$, $\nu[f | A]$, $\nu[f | \s]$
for the expectations of $f$ with respect to the measures $\nu(\cdot)$, $\nu(\cdot\mid A)$, $\nu(\cdot\mid\s)$, respectively, we obtain
\begin{align*}
\PSW f (\si)= \sum_{A \subseteq M(\sigma)} \nu(A \tc \si)\nu[\hat f \mid A] = \nu[\nu[\hat f\tc A]\tc \si]  = \nu[g\tc \si],
\end{align*}
where for ease of notation we set $g := \nu[\hat f\tc A]$.
Since $\mu[f] = 1$, we have $\mu[\PSW f] = 1$ and 
\begin{align*}
\ent_\mu( \PSW f) = \mu[ (\PSW f) \log(\PSW f)] = \mu\left[\nu\left[g\tc\si\right]\log (\nu\left[g\tc\si\right])\right].
\end{align*}
 The convexity of the function $x\cdot \log x$ and Jensen's inequality imply
 $$
\nu\left[g\tc\si\right]\log (\nu\left[g\tc\si\right]) \le \nu\left[g \log g \mid \s\right],
 $$
 and then, since $\nu[g]=\nu[\hat f]=\mu[f]=1$, we have
\begin{align}\label{eq:SWent}
\ent_\mu( \PSW f) \leq \mu\left[\nu\left[g\log g \mid \si\right]\right]
=  \nu[\nu[g\log g]\mid\sigma]
= \nu\left[g \log g\right] = \ent_\nu (g).
\end{align}

For any function $h: \Joint \mapsto \bbR_+$, we have by~\eqref{eq:total:entEdges} that
$ 
\Ent_\nu(h) = \Ent_\nu(\nu[h|A]) + \nu[\Ent_\nu(h|A)]
$. Hence,
$$
\Ent_\nu(\hat f) = \Ent_\nu(g) + \nu[\Ent_\nu(\hat f|A)],
$$
which by \eqref{eq:SWent} gives
$
\ent_\mu( \PSW f)  \le \ent_\nu(\hat f) -  \nu[\Ent_\nu(\hat f|A)].
$
The function $\hat f$ depends on $\si$ only, so $\ent_\nu(\hat f\tc \si)=0$. Therefore,
\begin{align*}
\ent_\mu( \PSW f) &\leq \ent_\nu(\hat f) - \nu\left[\ent_\nu(\hat f\tc A)+\ent_\nu(\hat f\tc \si) \right]. 
\end{align*}
The assumed spin/edge factorization \eqref{entfact2o-intro} then implies that
$
\ent_\mu( \PSW f) \leq (1-\delta)\ent_\nu(\hat f), 
$
with  $\delta=1/C$. Inequality~\eqref{relentsw-main} follows from the fact 
that $\ent_\nu(\hat f) = \ent_\mu(f)$.
\end{proof}
%

\begin{remark}\label{rem:boundary-entropy}
	We do not assume anything about the underlying graph in the previous proof, so Lemma~\ref{lem:main-intro-se-to-mixing}  holds for any graph $G$.
	In addition, our proof as stated applies to the Potts measure $\mu$ obtained as the marginal on spins of the joint measure $\nu$. If $\nu$ is as in~\eqref{jointnu-intro}, this yields only the Potts measure on $V$ with the {\em free} boundary condition. 
	However, the proof 
	 extends to the Potts measure with {\em any} boundary condition (or pinning of vertices) by choosing a spin-only boundary condition for $\nu$.
	In particular, Theorem~\ref{thm:intro:sw} holds for arbitrary boundary conditions,
	as stated in the introduction.
	For the special case when $G$ is a cube of $\bbZ^d$, we allow a slightly more general class of boundary conditions, involving both spin and edges, which we call \emph{admissible}; see Definition~\ref{def:admissible} and the examples immediately following it. 	
\end{remark}

\begin{remark}
	\label{rmk:log-sob}
	The entropy contraction established in~\eqref{relentsw-main} implies a modified log-Sobolev inequality, and can
	be viewed as a discrete time version of it; see Section \ref{subseq:prel}.
	The classical log-Sobolev constant, however, is not tight for the SW dynamics. 
	Indeed, the remark in \cite[Section 3.7]{Mart}
	shows a test function $f$ such that ${\var_\mu(\sqrt f)}/{\ent_\mu(f)} = O(n^{-1})$.
	Since $\mathcal D_{\PSW}(\sqrt f,\sqrt f) = \nu [ \var(\sqrt f \mid A) ]$, it follows from monotonicity of variance functional that 
	$\mathcal D_{\PSW}(\sqrt f,\sqrt f) \leq \var_\mu(\sqrt f)$ and so 
	$
	\frac{\mathcal D_{\PSW}(\sqrt f,\sqrt f)}{\ent_\mu(f)} = O(n^{-1})
	$  for this function.
\end{remark}

\section{Factorization of entropy in the joint space}\label{sec:factorization}

In this section, we prove our main technical result, Lemma~\ref{lem:main-intro}, 
which states that approximate even/odd factorization implies approximate spin/edge factorization
for the Potts measure on bipartite graphs.
%
For clarity of notation, and to simplify the proofs, we will restrict attention to $n$-vertex
cubes in $\bbZ^d$, but it should be clear that everything extends to arbitrary bipartite graphs
of constant degree with any spin-only boundary condition.
In addition, on $\bbZ^d$ we are able to extend our results to a more general class of boundary conditions in the joint space, involving both edges and vertices, that we call \emph{admissible}. 

\medskip\noindent\textbf{Admissible boundary conditions.} \
Let $\partial V$ be the set of vertices of $V$ with a neighbor in $\bbZ^d \setminus V$. 
Let $\partial \bbE$ denote the set of edges in $\bbE$ with at least one endpoint in $\partial V$. (Recall that $\bbE$ is the set of edges with both endpoints in $V$.)
We consider boundary conditions for the joint space 
on subsets $V_0 \subseteq \partial V$ and $\bbE_0 \subseteq \partial \EdgeSet$.
Specifically, we let $\psi: V_0 \mapsto [q]$ and $\phi: \bbE_0 \mapsto \{0,1\}$ 
and define 
\begin{equation}\label{jointnusieta}
\nu^{\psi,\phi}(\si,A) = \frac1{Z^{\psi,\phi}}\,p^{|A|}(1-p)^{|\bbE|-|A|}\IND(\si\sim A)\IND(\si\sim \psi)\IND(A\sim\phi),
\end{equation}
where $\si\sim A$ means that $A \subseteq M(\si)$, 
$\s \sim \psi$ that $\s$ and $\psi$ agree on the spins in~$V_0$, and
$A \sim \phi$ that $A$ and $\phi$ agree on the edges in~$\bbE_0$.  As usual, $Z^{\psi,\phi}$
is the corresponding partition function. 

\begin{definition}
	\label{def:admissible}
	We call the boundary condition {\em admissible} if
	$
	\bbE_0\subset \{\{u,v\} \in\partial\bbE:\, u\in V_0\};
	$
	that is, if all edges in $\bbE_0$ have at least one endpoint in $V_0$. 
\end{definition}
Notice that the free boundary condition ($V_0=\emptyset$ and $\bbE_0=\emptyset$) is admissible, and all spin-only boundary conditions ($V_0\subset \partial V$ and $\bbE_0=\emptyset$) are also admissible. In this case, the marginal on spins is just the Potts measure with $\psi$ as the boundary condition on $\partial V$ with $U^\psi=V_0$.
{For some additional examples of admissible boundary conditions see Section~\ref{sec:rc} and Figure~\ref{fig7.1}; in particular,~\eqref{rcsi} captures the effects an admissible boundary condition may have on the random-cluster marginal.

The main motivation for introducing the notion of admissible boundary conditions is that it guarantees that
the spin marginal of $\nu^{\psi,\varphi}$ has the desired exponential decay of correlations if the parameters $q$ and $\beta$ are such that SSM holds.
We shall see that all of our results concerning the joint measure and its dynamics 
on $\bbZ^d$ extend to the more general class of admissible boundary conditions. 
We can therefore restate Lemma~\ref{lem:main-intro} from the introduction 
for the special case of $\bbZ^d$ allowing arbitrary admissible boundary conditions.

\begin{lemma}\label{lemma:main-intro-ad}
	Let $\nu := \nu^{\psi,\phi}$ be the joint distribution with an admissible boundary condition $(\psi,\phi)$.
	Approximate even/odd factorization with constant $C$ of the spin marginal of $\nu$ implies that
	approximate spin/edge factorization holds with constant $C'=C'(C,d,q,\beta)$.
\end{lemma}

For simplicity, we will continue to write $\nu$ for the joint measure $\nu^{\psi,\phi}$ 
and $\mu$ for its marginal on spins.
We shall see that our proofs in this section are largely oblivious to the boundary condition or the geometry of $\bbZ^d$ (in fact, we only require the underlying graph to be bipartite). 
We also remark that, while we could allow a slightly more general family of boundary conditions than the admissible ones, some limitations are needed.  For instance, arbitrary edge boundary conditions may cause long-range dependencies; see, e.g.,~\cite{BSz2,BGV}. 
We proceed next with the proof of Lemma~\ref{lemma:main-intro-ad}.

\subsection{Proof of Lemma~\ref{lemma:main-intro-ad}}

\medskip\noindent
\textbf{Overview.}\ The following high level observations might be of help before entering the technical details of the proof. First, notice that the conclusion in the theorem would trivially hold true with constant $C=1$ if $\nu$ were a product measure with respect to the two sets of variables $(\sigma,A)$. This is a consequence of standard factorization properties of product measures. Thus, the minimal constant $C$ for which that statement holds is a measure  of the ``cost" for ``separating"  the two sets of variables. 

When the dependencies  between  the two sets of variables are very weak, a factorization statement could be obtained as in \cite{Cesi}. However, in our case the dependencies are not weak, since the spin variables interact locally with the edge variable in a strong way. For instance,  the presence of the edge $xy$ in $A$  forces deterministically the condition $\sigma_x=\sigma_y$. Thus, the fact that our statement holds with a constant $C$ independent of $n$ is highly nontrivial. 

On the other hand, for every $x\in V$ one can separate {\em locally} the two variables $(\sigma_x,A_x)$, where $A_x$ denotes the set of edge variables for edges incident to $x$,  by paying a finite cost $C$; this is the content of Lemma~\ref{lem:tensorx}  below. 
We can then lift this local factorization to a global factorization statement for the conditional measure $\nu(\cdot \mid \sigma_E)$, respectively $\nu (\cdot \mid \sigma_O)$, obtained by conditioning on the spin variables of all even vertices $E\subset V$, respectively of all odd vertices $O\subset V$. This is the content of Lemma~\ref{lem:tensor}. 

Lemma~\ref{lem:tensor} is the heart of the proof and relies crucially on the fact that  $\nu(\cdot \mid \sigma_E)$ is a product measure with respect to $\{(\sigma_x,A_x),  x\in O\}$, and $\nu (\cdot \mid \sigma_O)$ is a product measure with respect to $\{(\sigma_x,A_x),  x\in E\}$. Thus, we reduce  the problem of separating the spin/edge variables $(\sigma,A)$ to the problem of separating the even/odd spin variables $(\sigma_E,\sigma_O)$ for the joint distribution $\nu$. 
We then conclude by showing that even/odd factorization for the Potts measure $\mu$ implies
the even/odd factorization for $\nu$. 
This is the content of Lemma~\ref{lem:conc}.  

We now turn to the actual proof. 
Let $\nu(\cdot\tc\si_E,A)$ denote the measure $\nu$ conditioned on $\si_E=\{\si_v,\,v\in E\}$ and $A \subseteq \EdgeSet$. Similarly, $\nu(\cdot\tc\si_O,A)$ denotes the measure $\nu$ conditioned on $\si_O=\{\si_v,\,v\in O\}$ and $A$. 
We use $\ent_\nu(f\tc \si_E,A)$ and $\ent_\nu(f\tc \si_O,A)$ to denote the corresponding conditional entropies and $\nu\left[\ent_\nu (f\tc\si_E,A)\right]$, $\nu\left[\ent_\nu (f\tc \si_O,A)\right]$ for their expectations with respect to $\nu$. 
The next lemma shows that conditioning on the spin configuration of the even or the odd sub-lattice can only decrease the entropy of a function with respect to $\nu(\cdot\mid A)$.
 
\begin{lemma}
	\label{lem:conv}
	For all functions $f: \Joint \mapsto \R_+$ we have
	\begin{align*}
	&\nu\left[\ent_\nu (f\tc A)\right]\geq
	\nu\left[\ent_\nu (f\tc \si_E,A)\right];~\text{ and}\\
	&
	\nu\left[\ent_\nu (f\tc A)\right]\geq
	\nu\left[\ent_\nu (f\tc \si_O,A)\right].
	\end{align*}
\end{lemma}
\begin{proof}
	We can write
	\begin{align*}
	\nu\left[\ent_\nu (f\tc A)\right]
	& 
	= \nu\left[f \log \left(\frac{f}{\nu[f\tc A]}\right)\right]\\
	&=\nu\left[f\log \left(\frac{f}{\nu[f\tc \si_E,A]}\right)\right] + \nu\left[f\log \left(\frac{\nu[f\tc\si_E,A]}{\nu[f\tc A]}\right)\right] \\&
	= \nu\left[\ent_\nu (f\tc \si_E,A)\right] +  \nu\left[\nu[f\tc\si_E,A]\log \left(\frac{\nu[f\tc\si_E,A]}{\nu[f\tc A]}\right)\right]
	\\& = \nu\left[\ent_\nu (f\tc \si_E,A)\right] +\nu\left[\ent_\nu (\nu[f\tc\si_E,A]\tc A)\right]
	\\&\geq \nu\left[\ent_\nu (f\tc \si_E,A)\right].
	\end{align*}
	The same argument applies to the odd sites to deduce that 
	$\nu\left[\ent_\nu (f\tc A)\right]\geq
	\nu\left[\ent_\nu (f\tc \si_O,A)\right]$	
\end{proof}

The advantage of working with $\nu(\cdot\tc\si_O,A)$ or $\nu(\cdot\tc\si_E,A)$ instead of $\nu(\cdot\tc A)$ is that once we condition on the spins on all odd (resp.\ even) sites the measure becomes a product over the even (resp.\ odd) vertices, and we can exploit tensorization properties of entropy for product measures. 
The next lemma is a key step in the proof.
\begin{lemma}\label{lem:tensor}
	There exists a constant $\d_1>0$ depending only on $d,\beta,q$ such that, for all functions $f: \Joint \mapsto \R_+$, 
	\begin{align}
	&\nu\left[\ent_\nu (f\tc \si)\right] + \nu\left[\ent_\nu (f\tc \si_O,A)\right]
	\geq \d_1\,\nu\left[\ent_\nu (f\tc \si_O)\right],	\label{swent8a}
	\\
	&
	\nu\left[\ent_\nu (f\tc \si)\right] + \nu\left[\ent_\nu (f\tc \si_E,A)\right]
	\geq \d_1\,\nu\left[\ent_\nu (f\tc \si_E)\right].	\label{swent8b}
	\end{align}
\end{lemma}
We defer the proof of Lemma~\ref{lem:tensor} to later. Adding up 
\eqref{swent8a} and \eqref{swent8b} and using
 Lemma \ref{lem:conv} we obtain the estimate
\begin{align}\label{swent9}
\nu\left[\ent_\nu (f\tc \si) + 
\ent_\nu (f\tc A)\right]\geq\frac{\d_1}2\,\nu\left[\ent_\nu (f\tc \si_E) + \ent_\nu (f\tc \si_O)\right].
\end{align}
We then use a generalization of the entropy factorization from~\cite{CP} to reconstruct,
in the presence of approximate even/odd factorization, the global entropy $\ent_\nu(f)$ from the conditional average entropies $\nu\left[\ent_\nu (f\tc \si_E)\right]$ and $\nu\left[\ent_\nu (f\tc \si_O)\right]$ on the right hand side of \eqref{swent9}. 
\begin{lemma}\label{lem:conc}
Approximate even/odd factorization with consant $C$ implies that 
	for all functions $f: \Joint \mapsto \R_+$, 
	\begin{align*}
	\nu\left[\ent_\nu (f\tc \si_E) + \ent_\nu (f\tc \si_O)\right]\geq \d_2 \ent_\nu(f),
	\end{align*}
	where $\d_2=1/C$.	
\end{lemma}
\begin{proof}
	We need the following observations:
	\begin{align}
	&\ent_\nu (f\tc \si_O)= 
	\ent_\nu\left(\nu\left[f\tc \si\right]\tc \si_O\right) + \nu\left[\ent_\nu (f\tc \si)\tc \si_O\right] ,
	\label{swent100a}\\
	&
	\ent_\nu (f\tc \si_E)= 
	\ent_\nu\left(\nu\left[f\tc \si\right]\tc \si_E\right) + \nu\left[\ent_\nu (f\tc \si)\tc \si_E\right]  \label{swent100b}.
	\end{align}
	Indeed, to establish~\eqref{swent100a} note that from the definition of conditional entropy we get
	\begin{align*}
	\ent_\nu (f\tc \si_O)
	&= \nu\left[f\log \left(\frac{f}{\nu[f\tc \si_O]}\right) \,\middle|\, \si_O\right]\\
	&=\nu\left[f\log \left(\frac{f}{\nu[f\tc \si_E,\si_O]}\right)\,\middle|\, \si_O\right] + \nu\left[f\log \left(\frac{\nu[f\tc\si_E,\si_O]}{\nu[f\tc \si_O]}\right)\,\middle|\,\si_O\right] \\
	&=\nu\left[f\log \left(\frac{f}{\nu[f\tc \si_E,\si_O]}\right)\,\middle|\, \si_O\right] + \nu\left[ \nu[f \mid \sigma]\log \left(\frac{\nu[f\tc\si]}{\nu[f\tc \si_O]}\right)\,\middle|\,\si_O\right] \\
	&= \nu\left[\ent_\nu (f\tc \si)\tc \si_O\right] + \ent_\nu\left(\nu\left[f\tc \si\right]\tc \si_O\right),
	\end{align*}
	where we also use the fact that $\nu[\cdot\tc \si_E,\si_O]=\nu[\cdot\tc \si]$. 
	The same argument applies to \eqref{swent100b}. 
	
	Now, since the function 	
	$\nu\left[f\tc \si\right]$ depends only on the spin configuration $\si$,  
	$$
	\nu\left[\ent_\nu (\nu[f\tc \si]\tc \si_E) + \ent_\nu (\nu[f\tc \si]\tc \si_O)\right] 
	= 
	\mu\left[\ent_\mu (\nu[f\tc \si]\tc \si_E) + \ent_\mu (\nu[f\tc \si]\tc \si_O)\right]; 
	$$
	and we may apply the approximate even/odd factorization
	to the function $\nu\left[f\tc \si\right]$.  
	Then, there exists a constant $\d_2\in(0,1]$ such that
	\begin{align}
	\label{swent011}
	\mu\left[\ent_\mu (\nu[f\tc \si]\tc \si_E) + \ent_\mu (\nu[f\tc \si]\tc \si_O)\right] 
	& \geq 
	\d_2\, \ent_\mu \left(\nu\left[f\tc \si\right]\right).
	\end{align}
	Therefore, observing that 
	$$
	\nu\left[
	\nu\left[\ent_\nu (f\tc \si)\tc \si_O\right]+  \nu\left[\ent_\nu (f\tc \si)\tc \si_E\right]\right]=2\,\nu\left[\ent_\nu (f\tc \si)\right],
	$$
	we obtain from~\eqref{swent100a},~\eqref{swent100b} and~\eqref{swent011}
	\begin{align*}
	\nu\left[\ent_\nu (f\tc \si_E) + \ent_\nu (f\tc \si_O)\right]
	&
	\geq 
	\d_2\, \ent_\nu \left(\nu\left[f\tc \si\right]\right) + 2\,\nu\left[\ent_\nu (f\tc \si)\right].
	\end{align*}
	Since $\d_2\leq 1$, the standard decomposition in~\eqref{eq:total:entSpins}
	implies 
	\begin{align*}
	\nu\left[\ent_\nu (f\tc \si_E) + \ent_\nu (f\tc \si_O)\right]
	\geq 
	\d_2\,\ent_\nu( f),
	\end{align*}
	as claimed.
\end{proof}

\subsection{Proofs of main results}

The proofs of Lemma~\ref{lem:main-intro} and Theorems~\ref{thm:intro:sw} and~\ref{thm:evenodd-comparison} are now immediate. 

\begin{proof}[Proof of Lemma~\ref{lem:main-intro}]
	We note that inequality \eqref{swent9} and Lemma~\ref{lem:conc}
	are valid for any bipartite graph of bounded degree; the result follows by taking $C=2/\d_1\d_2$.
\end{proof}

\begin{proof}[Proof of Theorem~\ref{thm:evenodd-comparison}]
It follows immediately from Lemmas~\ref{lem:main-intro} and~\ref{lem:main-intro-se-to-mixing}.
\end{proof}

We now also prove our main theorem (Theorem~\ref{thm:intro:sw}).  We use the following result of~\cite{CP} that under SSM the
even/odd factorization holds.

\begin{theorem}[Theorem~4.3 in~\cite{CP}]
	\label{thm:cp}
    SSM implies that there exists a constant $\delta>0$ such that for all cubes of $\bbZ^d$,
    all boundary conditions, and
    for all functions $f: \Omega \mapsto \bbR_+$,
	\begin{align*}
	\mu\left[\ent_\mu (f\tc \si_E) + \ent_\mu (f\tc \si_O)\right] 
	& \geq 
	\delta\, \ent_\mu \left(f\right).
	\end{align*}
\end{theorem}

\begin{proof}[Proof of Theorem~\ref{thm:intro:sw}]
	Theorem~\ref{thm:cp} \cite[Theorem 4.3]{CP} implies that the even/odd factorization
	holds for any boundary condition whenever SSM holds.
	Then, from Lemma~\ref{lemma:main-intro-ad} we know that approximate spin/edge factorization holds;
	the result then follows from applying Lemmas~\ref{lem:main-intro} and~\ref{lem:main-intro-se-to-mixing}.
\end{proof}


It remains for us to provide the proof of Lemma~\ref{lem:tensor}, which we do in the next subsection.

\subsection{Proof of Lemma \ref{lem:tensor}}

Before giving the proof of Lemma \ref{lem:tensor}, we mention several useful facts about the joint distribution~$\nu$.
The first key fact is that, for any fixed configuration $\si_O$ of spins on the odd sub-lattice, the conditional probability $\nu(\cdot\tc\si_O)$ is a product measure. 
That is,
\begin{align}\label{eo1}
\nu(\cdot\tc\si_O)
= \bigotimes_{x\in E}\nu_x(\cdot\tc\si_O),
\end{align}
where,
for each $x\in E$, $\nu_x(\cdot\tc\si_O)$ is the probability measure on $\{1,\dots,q\}\times \{0,1\}^{\deg(x)}$,
where $\deg(x)$ denotes the degree of $x$,  
 described as follows: pick the spin of site $x$ according to the Potts measure on $x$ conditioned on the spin of its neighbors in $\s_O$; then, independently for every edge $xy \in \EdgeSet$ incident to the vertex $x$, if $\si_x=\si_y$ set $A_{xy}=1$ 
with probability $p$ and set $A_{xy}=0$ otherwise; if $\si_x\neq\si_y$, set $A_{xy}=0$.  
(Note that in this section, to simplify notation, we shall use $xy$ to denote the edge $\{x,y\}$,
and view the edge configuration $A$ as a vector in $\{0,1\}^\EdgeSet$.)

Consider now the measure $\nu(\cdot\tc\si_O,A)$ obtained by further conditioning on a valid configuration of all edge variables $A$. Here $A$ is {\it valid\/} if it is compatible with the fixed spins $\si_O$. 
This is again a product measure; namely
\begin{align}\label{eo11}
\nu(\cdot\tc\si_O,A)
= \bigotimes_{x\in E}\nu_x(\cdot\tc\si_O,A),
\end{align}
where $\nu_x(\cdot\tc\si_O,A)$ is the probability measure on $\{1,\dots,q\}$ that is uniform if $x$ has no incident edges in~$A$, and is concentrated on the unique admissible value given $\si_O$ and $A$ otherwise.

Next, we note that $\nu(\cdot\tc\si)$ is a product of Bernoulli($p$) random variables over all monochromatic edges in $\si$, while it is concentrated on $A_e=0$ on all remaining edges. Therefore we may write
\begin{align}\label{eo12}
\nu(\cdot\tc\si)
= \bigotimes_{x\in E}\nu_x(\cdot\tc\si),
\end{align}
where $\nu_x(\cdot\tc\si)$ is the probability measure on $\{0,1\}^{\deg(x)}$ given by  the product of Bernoulli($p$) variables on all edges $xy$ incident to $x$ such that $\si_x=\si_y$ and is concentrated on $A_{xy}=0$ if $\si_x\neq\si_y$. 

We write 
$\ent_x(\cdot\tc \si_O)$, $\ent_x(\cdot\tc \si_O,A)$, $\ent_x(\cdot\tc \si)$
for the entropies with respect to the distributions $\nu_x(\cdot\tc \si_O)$, $\nu_x(\cdot\tc \si_O,A)$, $\nu_x(\cdot\tc \si)$ respectively. 
The first observation is that, for every site $x$, there is a local factorization of entropies in the following sense. 
\begin{lemma}\label{lem:tensorx}
	There exists a constant $\d_1>0$ such that, for all functions $f\geq 0$ and all $x\in E$,
	\begin{align}\label{swent18}
	\nu_x\left[\ent_x (f\tc \si)\tc \si_O\right] + \nu_x\left[\ent_x (f\tc \si_O,A)\tc \si_O\right]
	\geq \d_1\,\ent_x (f\tc \si_O).
	\end{align}
\end{lemma}
\begin{proof}
	For $x \in V$, let $A_x$ be random variable in $\{0,1\}^{\deg(x)}$ corresponding to the configuration of the edges incident to $x$ in $A$.
	If we replace entropy by variance, then \eqref{swent18} is a spectral gap inequality for the Markov chain where the variable $(\si_x,A_x)\in [q]\times \{0,1\}^{\deg(x)} =: \mathcal S$ is updated as follows. At each step, with probability $1/2$ the spin $\si_x$ is updated with a sample from $\nu_x(\cdot\tc\si_O,A)$, and with probability $1/2$ the edges $A_x$ incident to $x$ are simultaneously updated with a sample from $\nu_x(\cdot\tc\si)$. 
	Let $P_x = \frac{Q_x + S_x}{2}$ denote the transition matrix of this Markov chain, where 
	$Q_x$, $S_x$ are the stochastic matrices corresponding to the spin and edge moves at $x$, respectively.
	Let $\mathcal D_{P_x}$, $\mathcal D_{Q_x}$ and $\mathcal D_{S_x}$ denote the corresponding Dirichlet forms.
	Observe that, by updating first the edges with an empty configuration and then the spin, two arbitrary initial configurations can be coupled after two steps with probability at least $\frac14(1-p)^{-2d}$, and thus for any function $f: \mathcal S \mapsto \bbR_+$
	$$
	\frac{\mathcal D_{Q_x}(f,f)+\mathcal D_{S_x}(f,f)}{2} = \mathcal D_{P_x}(f,f) \ge \d_0 \var_x (f\tc \si_O),
	$$
    where $\d_0>0$ is a constant depending only on $p$ and $d$. 
	Using the standard facts that
	$\mathcal D_{Q_x}(f,f) = \nu_x\left[\var_x (f\tc \si)\tc \si_O\right]$ and $\mathcal D_{S_x}(f,f) = \nu_x\left[\var_x (f\tc \si_O,A)\tc \si_O\right]$, we arrive at the inequality
	\begin{align}\label{swent19}
	\frac {\nu_x\left[\var_x (f\tc \si)\tc \si_O\right] +\nu_x\left[\var_x (f\tc \si_O,A)\tc \si_O\right]}{2}
	\geq \d_0\var_x (f\tc \si_O).
	\end{align}

	A well known general relation between entropy and variance 
	(see, e.g., Theorem A.1 and Corollary A.4 in \cite{DSC}) shows that, for all $f\geq 0$,
	\begin{equation}
	\label{rough1}
	\ent_x (f\tc \si_O)\leq C_1\var_x (\sqrt f\tc \si_O),
	\end{equation}
	where $C_1=C_1(q,p,d)$ is a constant independent of $n$, since
	we are considering the conditional measure at the single site $x$.
	Thus, 
	applying \eqref{swent19} to $\sqrt f$ instead of $f$, we obtain
	\begin{align}\label{swent20}
   	\frac {\nu_x\left[\var_x (\sqrt{f}\tc \si)\tc \si_O\right] +\nu_x\left[\var_x (\sqrt{f}\tc \si_O,A)\tc \si_O\right]}{2} \geq 	\frac{\d_0}{C_1}\,\ent_x (f\tc \si_O). 
	\end{align}
	The conclusion \eqref{swent18} follows by recalling that for any $f\geq 0$ the variance of $\sqrt f$ is at most the entropy of $f$ for any underlying probability measure; see, e.g., \cite[Lemma~1]{Latala}. 
	In particular, 
	$\var_x (\sqrt f\tc \si)\leq \ent_x (f\tc \si)$ and $\var_x (\sqrt f\tc \si_O,A)\leq \ent_x (f\tc \si_O,A)$.
\end{proof}

To prove Lemma \ref{lem:tensor}, we need to lift the inequality of Lemma \ref{lem:tensorx} to the product measure 
$\nu(\cdot\tc\si_O)=\otimes_{x\in E}\nu_x(\cdot\tc\si_O)$.

\begin{proof}[Proof of Lemma~\ref{lem:tensor}]
We will prove~\eqref{swent8a}; exactly the same argument applies to~\eqref{swent8b}. 
Let $x=1,\dots, w$ denote an arbitrary ordering of the even sites $x\in E$.
Let $A_x \in \{0,1\}^{\deg(x)}$ be the random variable corresponding to the state of the edges incident to $x$. We write $\xi_x=(\si_x,A_x)$ for the pair of variables at $x$. 
We first observe that
\begin{align}\label{swent21}
\ent_\nu(f\tc\si_O) = \sum_{x=1}^w
\nu\left[ \ent_x (g_{x-1}\tc\si_O)\tc \si_O\right],
\end{align}
where $g_x=\nu\left[f\tc \si_O,\xi_{x+1},\dots,\xi_{w}\right]$, so that 
$g_0=f$ and $g_{w}=\nu\left[f\tc \si_O\right]$. To prove \eqref{swent21}, we note that since $\nu(\cdot\tc\si_O)=\otimes_{x\in E}~\nu_x(\cdot\tc\si_O)$, one has 
$\nu_x[g_{x-1}\tc \si_O]=g_{x}.$ 
Therefore,
\begin{align}
\ent_\nu(f\tc\si_O) &= \nu\left[g_0
\log\left(g_0/g_{w}\right)\tc\si_O\right]
=\sum_{x=1}^{w}  \nu\left[g_{0}\log \left(g_{x-1}/g_{x}\right)\tc\si_O\right]\notag.
\end{align}
Since the $g_x$ are (conditional) expectations, we deduce
\begin{align}
\ent_\nu(f\tc\si_O) 
&=\sum_{x=1}^{w}  \nu\left[g_{x-1}\log \left(g_{x-1}/g_{x}\right)\tc\si_O\right] \notag\\
&=\sum_{x=1}^{w}  \nu\left[\nu_x\left[g_{x-1}\log \left(g_{x-1}/g_{x}\right)\tc\si_O\right]\tc\si_O\right]\notag\\
&=\sum_{x=1}^{w}\nu\left[\ent_{x} (g_{x-1}\tc\si_O)\tc\si_O\right]. \label{tele1}
\end{align}
From \eqref{tele1}, using Lemma 
\ref{lem:tensorx} we obtain
\begin{align}
\d_1\,\ent_\nu (f\tc \si_O)&\leq 
\sum_{x=1}^w
\nu\left[\nu_x\left[\ent_x (g_{x-1}\tc \si)\tc \si_O\right] + \nu_x\left[\ent_x (g_{x-1}\tc \si_O,A)\tc \si_O\right]
\tc \si_O\right]\notag\\&
= \sum_{x=1}^w
\nu\left[\ent_x (g_{x-1}\tc \si) + \ent_x (g_{x-1}\tc \si_O,A)
\tc \si_O\right].
\label{swent22}
\end{align}
Observe that 
$\sum_{x=1}^w \nu\left[\ent_x (g_{x-1}\tc \si) \mid \sigma_O \right]$ and
$\sum_{x=1}^w \nu\left[\ent_x (g_{x-1}\tc \si_O,A)\tc \si_O\right]$ are
``tensorized'' versions of 
$\nu\left[\ent_\nu (f\tc \si)\tc \si_O\right]$
and 
 $\nu\left[\ent_\nu (f\tc \si_O,A)\right]$, respectively, which are the terms on the right hand side of~\eqref{swent8a}.
Using similar but somewhat more involved ideas to those used to derive \eqref{tele1}, we can establish the following.

\begin{lemma}
	\label{lemma:iterative:bounds}
	 \
\begin{enumerate}
\item $\sum_{x=1}^w
\nu\left[\ent_x (g_{x-1}\tc \si)\tc \si_O\right]\leq \nu\left[\ent_\nu (f\tc \si)\tc \si_O\right]$
\item $\sum_{x=1}^w
\nu\left[ \ent_x (g_{x-1}\tc \si_O,A)
\tc \si_O\right]
\leq \nu\left[\ent_\nu (f\tc \si_O,A)\tc \si_O\right]$
\end{enumerate}
\end{lemma}
Before providing the proof of this lemma, we finish the proof of Lemma~\ref{lem:tensor}. 
Inequality~\eqref{swent22} together with parts 1 and 2 of Lemma~\ref{lemma:iterative:bounds}
show that
\begin{align}\label{swentop}
\d_1\,\ent_\nu (f\tc \si_O)
\leq  \nu\left[\ent_\nu (f\tc \si)\tc \si_O\right] + \nu\left[\ent_\nu (f\tc \si_O,A)\tc \si_O\right].
\end{align}
Taking expectations with respect to $\nu$ in \eqref{swentop} we arrive at \eqref{swent8a} and the proof is complete.
\end{proof}

We finish the proof of Lemma~\ref{lem:tensor} 
by providing the proof of Lemma~\ref{lemma:iterative:bounds}.

\begin{proof}[Proof of Lemma~\ref{lemma:iterative:bounds}]
	
We start with part 2. Let $h_x=\nu\left[f\tc \si_O,\si_{x+1},\dots,\si_{w},A\right]$, so that $h_0=f$ and $h_{w}=\nu\left[f\tc \si_O,A\right]$. Since $\nu(\cdot\tc \si_O,A)$ is a product measure, 
$\nu_x[h_{x-1}\tc \si_O,A]=h_{x}.$
Therefore, reasoning as in \eqref{swent21} we obtain
\begin{align}\label{swent25}
\ent_\nu(f\tc\si_O,A) = \sum_{x=1}^w
\nu\left[ \ent_x (h_{x-1}\tc\si_O,A)\tc \si_O,A\right].
\end{align}
Taking  
expectations with respect to $\nu(\cdot\tc\si_O)$ in \eqref{swent25} we see that it is
sufficient to show that, for all $x$,
\begin{align}\label{swent26}
\nu\left[ \ent_x (g_{x-1}\tc\si_O,A)\tc \si_O\right]\leq
\nu\left[ \ent_x (h_{x-1}\tc\si_O,A)\tc \si_O\right].
\end{align}
To prove \eqref{swent26}, we introduce the measures $\mu_k=\otimes_{x=1}^k\nu_x(\cdot\tc\si_O)$ and $\mu_k^A=\otimes_{x=1}^k\nu_x(\cdot\tc\si_O,A)$. Then we have $g_x = \mu_x[ f]$, $h_x= \mu_x^A [f]$, and $g_x=\mu_x [h_x]$.
Also, we simplify the notation by writing 
$\nu_x(\cdot\tc\si_O,A)=:\nu_x^A$.
Now the product structure 
implies the commutation relation 
between expectations
\begin{align}\label{swent27}
\nu_x^A g_{x-1} 
=    \nu_x^A \mu_{x-1} h_{x-1}
=  \mu_{x-1} \nu_x^A h_{x-1}.
\end{align}
Therefore,
\begin{align}
\nu\left[ \ent_x (g_{x-1}\tc\si_O,A)\tc \si_O\right]
&= 
\nu\left[g_{x-1}  \log\left(g_{x-1} /\nu_x^A g_{x-1}\right)\tc \si_O\right]
\notag\\&=
\nu\left[ \mu_{x-1} h_{x-1} \log\left(\mu_{x-1}h_{x-1} /\mu_{x-1}\nu_x^A h_{x-1}\right)\tc \si_O\right]
\notag\\&=
\nu\left[ h_{x-1} \log\left(\mu_{x-1}h_{x-1} /\mu_{x-1}\nu_x^A h_{x-1}\right)\tc \si_O\right]
\notag\\ &=
\nu\left[\nu_x^A\left(h_{x-1} \log\left(g_{x-1} /\nu_x^A g_{x-1}\right)\right)\tc \si_O\right].\label{paris44}
\end{align}
From the variational principle \eqref{varprin} it follows that
\begin{equation}
\label{varprin1}
\nu_x^A\left[h_{x-1} \log\left(g_{x-1} /\nu_x^A[g_{x-1}]\right)\right]\leq \ent_x (h_{x-1}\tc\si_O,A),
\end{equation}
which combined with \eqref{paris44} proves \eqref{swent26}. This completes the proof of part 2.

We use a similar argument for part 1. Let $\psi_x=\nu\left(f\tc \si,A_{x+1},\dots,A_{w}\right)$, so  
that $\psi_0=f$ and $\psi_{w}=\nu\left(f\tc \si\right)$. Notice that 
$\nu_x[\psi_{x-1}\tc \si]=\psi_{x}.$
Therefore, as in \eqref{swent21}, 
\begin{align*}
\ent_\nu(f\tc\si) = \sum_{x=1}^w
\nu\left[ \ent_x (\psi_{x-1}\tc\si)\tc \si\right].
\end{align*}
Taking 
expectations with respect to $\nu(\cdot\tc\si_O)$ we see that it is sufficient  to show that, for all $x\in E$,
\begin{align}\label{swent261}
\nu\left[ \ent_x (g_{x-1}\tc\si)\tc \si_O\right]\leq
\nu\left[ \ent_x (\psi_{x-1}\tc\si)\tc \si_O\right].
\end{align}
Introducing the measures $\mu_k=\otimes_{x=1}^k\nu_x(\cdot\tc\si_O)$,  $\mu_k^\si=\otimes_{x=1}^k\nu_x(\cdot\tc\si)$, and $\nu_x^\si=\nu_x(\cdot\tc\si)$, we have $g_x = \mu_x[f]$, $\psi_x= \mu_x^\si[f]$, and $g_x=\mu_x [\psi_x]$.
As in~\eqref{swent27}, we have the commutation relation 
\begin{align*}
\nu_x^\si g_{x-1} = \nu_x^\si \mu_{x-1} \psi_{x-1}
=  \mu_{x-1} \nu_x^\si \psi_{x-1}.
\end{align*}
Therefore, as in \eqref{paris44}-\eqref{varprin1} we obtain
\begin{align*}
\nu\left[ \ent_x (g_{x-1}\tc\si)\tc \si_O\right]
&= 
\nu\left[g_{x-1}  \log\left(g_{x-1} /\nu_x^\si g_{x-1}\right)\tc \si_O\right]
\\&=
\nu\left[ \mu_{x-1} \psi_{x-1} \log\left(\mu_{x-1}\psi_{x-1} /\mu_{x-1}\nu_x^\si \psi_{x-1}\right)\tc \si_O\right]
\\&=
\nu\left[ \psi_{x-1} \log\left(\mu_{x-1}\psi_{x-1} /\mu_{x-1}\nu_x^\si \psi_{x-1}\right)\tc \si_O\right]
\\ &=
\nu\left[\nu_x^\si\left(\psi_{x-1} \log\left(g_{x-1} /\nu_x^\si g_{x-1}\right)\right)\tc \si_O\right]
\\ &\leq \nu\left[\ent_x (\psi_{x-1}\tc\si)\tc \si_O\right].
\end{align*}
This proves~\eqref{swent261} and completes the proof of part 1.
\end{proof}

\section{A lower bound for the SW dynamics}
\label{sec:lower-bound}

In this section we establish an asymptotically tight lower bound for the mixing time of SW dynamics whenever SSM holds; this result implies the lower bound in Theorem~\ref{thm:intro:sw} from the introduction.
\begin{theorem}
	\label{thm:lb}
	In an $n$-vertex cube of $\bbZ^d$,
	for all integer $q \ge 2$ and all $\beta > 0$, SSM implies that for all boundary conditions $\tmix(SW) = \Omega(\log n)$.
\end{theorem}

The main new ingredient in the proof of this result is a bound on the speed of propagation of disagreements under a coupling of the steps of the SW dynamics provided SSM holds.
With this new tool, we are able to adapt the lower bound framework of Hayes and Sinclair~\cite{HS}
for the Glauber dynamics to the SW setting.  
We also use a recently established fact about concentration 
properties of the Potts measure due to~\cite{DRT}.

\bigskip\noindent\textbf{SW coupling.} \ Consider two copies of the SW dynamics on the graph $G=(V,\EdgeSet)$, where $V$ is an $n$-vertex cube of $\bbZ^d$.
Let ${X_t}$ and ${Y_t}$ be the configurations of these copies at time $t \ge 0$.
We can couple the steps of the SW dynamics as follows:
\begin{enumerate}
	\item Draw 
	$|\EdgeSet|$ independent, uniform random numbers from $[0,1]$, one for each edge. Let $r_e(t) \in [0,1]$ denote the random number corresponding to the edge $e \in \EdgeSet$.
	\item Draw $|V|$ independent, uniform random numbers from $\{1,...,q\}$, one for each vertex. Let $s_v(t) \in \{1,...,q\}$ denote the random number for $v \in V$.
	\item Let $A_X = \{e \in M(X_t): r_e(t) \le p\}$ and $A_Y = \{e \in M(Y_t): r_e(t) \le p\}$,
	where recall that $M(X_t)$ and $M(Y_t)$ denote the set of monochromatic edges in $X_t$ and $Y_t$, respectively 
	\item For each connected component $\mathcal C$ of $(V,A_X)$ or $(V,A_Y)$,
	we let $s_{\mathcal C} = s_v(t)$, where $v$ is the vertex in $\mathcal C$ with the smallest coordinate sum. (If two or more vertices in $\mathcal C$ have the same coordinate sum, we break ties ``lexicographically'' using the coordinates.) Then, every vertex of $\mathcal C$ is assigned the spin $s_{\mathcal C} $.
\end{enumerate}

The key property of the SW coupling is that, after assigning the edges, two identical connected components in $A_X$ and $A_Y$ will be assigned the same spin (namely, the spin~$s_v$ of their common vertex~$v$
with smallest coordinate sum). 
We show that, under SSM, the SW coupling propagates disagreements slowly for a suitable starting condition. 

To describe our starting condition we introduce the notion of \emph{$L$-shattered} configurations.
\begin{definition}
	Consider the graph $G=(V,\EdgeSet)$, where $V$ is an $n$-vertex cube of $\bbZ^d$.
	For a configuration $\sigma$ 
	on $V$, let $A_\sigma \subseteq M(\sigma)$
	be the configuration that results from keeping each
	monochromatic edge in $M(\sigma)$ independently 
	with probability $p = 1 - \exp(-\beta)$. 
	We say that $\sigma$ is \emph{$L$-shattered} in $V$
	if,
	with probability at least $1-|V|\exp(-\gamma L)$ where $\gamma > 0$ is a fixed constant we choose later,
	for every $v \in |V|$ at distance at least 
	$2L$ from the boundary of $|V|$, 
	the connected component of $v$ in $A_\sigma$ 
	does not reach the boundary of the cube $\Lambda_v(L)$ centered at $v$ of side length $L$.
\end{definition}
Note that the above defined notion involves 
a probability that decays exponentially with $L$, so
the dimension of the cube $V$ will not be as significant as long as $\log n \ll L\ll n$. 
The following lemma establishes a concentration of the probability mass on shattered configurations under SSM (for the monochromatic ``all 1'' boundary condition).

\begin{lemma}
	\label{lemma:sht:stat}
	Let $\mathfrak S$ be the set of $L$-shattered configurations 
	of the $n$-vertex cube $G=(V,\EdgeSet)$ of $\bbZ^d$.
	There exists a constant $c > 0$ such that
	for all integers $q \ge 2$ and $L \ge 1$, SSM implies that 
	$\mu^1 (\mathfrak S) \ge 1 - \exp(-cL).$
\end{lemma}

The proof of this lemma, which follows straightforwardly from the results in~\cite{DRT}, 
 will be provided later in Section~\ref{subsec:aux:lb}.
We can now describe our starting condition for the SW dynamics.

\bigskip\noindent\textbf{A starting condition.} \
We consider a regular pattern of non-overlapping $d$-dimensional cubes of side
length $\ell = (\log n)^{3}$ with a fixed minimal distance between cubes.
Formally, consider the cubes of side length $\ell$ centered at $(\ell+3) \cdot \vec{h}$ where $\vec{h} \in \bbZ^d$. These cubes have volume $\ell^d$ and are at distance $4$ from each other. We let $B_1,B_2,\dots B_N \subset V$ be the collection of those cubes that are contained in $V$ and at distance at least $4$ from the boundary $\partial V$; then, $N = \Theta({n}/{\ell^d})$.

Let $B = \bigcup_{i=1}^N B_i$, $\partial B = \bigcup_{i=1}^N \partial B_i$
and let $e_i$ be an edge at the center of $B_i$.
For definiteness, we may assume that $\ell$ is odd so that there is a unique vertex $v_i$ at the center of each $B_i$; we take $e_i = \{v_i,u_i\}$ where 
$u_i = v_i + (1,0,\dots,0) \in \bbZ^d$.
Let $\mathcal A_i$ be the set of configurations on $B_i$
in which the spins at the endpoints of $e_i$ are the same,
and let $\mathcal S_i$ be the set of $L$-shattered configurations in each $B_i$.
(Later we will set $L = C (\log n)$ with $C > 0$ a large constant.)	

We consider two variants of the SW dynamics, $\{X_t\}$ and $\{Y_t\}$, with the same initial condition $X_0=Y_0$.
The chain $\{X_t\}$ is an instance of the standard SW dynamics on $(V,\EdgeSet)$;
for the initial state $X_0$ of $\{X_t\}$ we set
the spins of all the vertices in $U = (V \setminus B) \cup \partial B$ to $1$.
The configuration 
in each cube $B_i$ is sampled (independently) proportional to $\mu_{B_i}^1$ on $\mathcal S_i \cap \mathcal A_i$, where $\mu_{B_i}^1$ denotes 
the Potts measure on $B_i$ with the ``all $1$'' monochromatic boundary condition.

The other instance we consider, $\{Y_t\}$, only updates the spins of the vertices in $B\setminus \partial B$. 
That is, after adding all the monochromatic edges independently with probability $p=1-\exp(-\beta)$, only the connected components fully contained in $B$ update their spins.
(Note that if a component touches the boundary of $B$, then it is not updated since the boundary is frozen to the spin $1$ by the boundary condition.)   
We set $Y_0 = X_0$
and couple the evolution of $Y_t$ and $X_t$ using the SW coupling defined earlier.
We can view $\{Y_t\}$ as a dynamics on the configurations on $B$ whose stationary measure is $\mu_B^1 = \otimes_{i=1}^N \mu_{B_i}^1$.
We also observe that a step of $\{Y_t\}$ is equivalent to performing one step of the SW dynamics in each $B_i$ independently.

Note that $X_0=Y_0$, and any disagreements between $X_t,Y_t$ at later times~$t$
can arise only from the fact that $Y_t$ does not update the spins outside~$B$: i.e., 
disagreements must propagate into the~$B_i$ from their boundaries.  
The following result, whose proof we defer until after the proof of 
Theorem~\ref{thm:lb}, provides a bound on the speed of propagation of these disagreements
under the SW coupling with the specified initial condition.
In particular it says that, for $t=\Omega(\log n)$ steps, $X_t,Y_t$ agree w.h.p.\ on the spins
at the center of every cube~$B_i$.

\begin{theorem}
	\label{thm:dp}
	Let $\mathcal C = \bigcup_{i=1}^N e_i$
	and set $L = C (\log n)$.
	For any constant $A > 0$, for a sufficiently large constant $C > 0$
	SSM implies that 
	$$
	\Pr\left[\forall t \le A \log n:~X_t(\mathcal C) = Y_t(\mathcal C)\right] = 1-o(1).
	$$
\end{theorem} 

A key ingredient in the proof of Theorem~\ref{thm:dp} (and also of Theorem~\ref{thm:lb}) is the following discrete time version of the completely monotone decreasing (CMD) property of reversible Markov chains from~\cite{HS}; the proof of this lemma is provided in Section~\ref{subsec:aux:lb}.

\begin{lemma}\label{lem:HS}
	Let $\{X_t\}$ denote a discrete time Markov chain with finite state space $\Omega$, reversible with respect to $\pi$ and with a positive semidefinite transition matrix. 
	Let $B\subset \Omega$ denote an event. 
	If $X_0$ is sampled proportional to $\pi$ on $B$, then 	$\Pr(X_t\in B)\geq \pi(B)$ for all $t \ge 0$, and for all $t \ge 1$ 
	\begin{equation*}\label{eq:bo1}
	\Pr(X_t\in B)\geq \pi(B) + (1-\pi(B))^{-t+1}(\Pr(X_1\in B)-\pi(B))^{t}.
	\end{equation*}
\end{lemma}

We now proceed with the proof of Theorem~\ref{thm:lb}.

\begin{proof}[Proof of Theorem~\ref{thm:lb}]

	Our goal is to show that at some time $T = \Theta(\log n)$
	$$
	{\|X_T - \mu\|}_\textsc{tv} > \frac 12,
	$$
	where with a slight abuse of notation we use $X_T$ for the distribution of the chain $\{X_t\}$ at time $T$.
	This clearly implies that the mixing time of the SW dynamics is $\Omega(\log n)$.
	
	Let $\mathcal C = \bigcup e_i$ and let $e_i = \{a_i,b_i\}$. 
	Let $\hat\mu_{\mathcal C}$ and $\hat\mu_{\mathcal C}^1$ be the marginals of $\mu$ and $\mu^{1}_B$, respectively, on $\mathcal C$. 
	Then,
	\begin{align}
	{\|X_T - \mu\|}_\textsc{tv} 
	&\ge {\|X_T(\mathcal C) - \hat\mu_{\mathcal C}\|}_\textsc{tv} \notag\\
	&\ge {\|Y_T(\mathcal C) - \hat\mu_{\mathcal C}\|}_\textsc{tv} - {\|X_T(\mathcal C) - Y_T(\mathcal C)\|}_\textsc{tv} \notag\\
	&\ge {\|Y_T(\mathcal C) - \hat\mu_{\mathcal C}^1\|}_\textsc{tv} - {\|\hat\mu_{\mathcal C}^1- \hat\mu_{\mathcal C}\|}_\textsc{tv} -
	{\|X_T(\mathcal C) - Y_T(\mathcal C)\|}_\textsc{tv}. \label{eq:main-trig}
	\end{align}
	We bound each term of~\eqref{eq:main-trig} independently.
	We note first that by Theorem~\ref{thm:dp}
	\begin{align*}
	{\|X_T(\mathcal C) - Y_T(\mathcal C)\|}_\textsc{tv} 
	\le \Pr(X_T(\mathcal C) \neq Y_T(\mathcal C)) 
	& = o(1).
	\end{align*}
	
	We proceed to bound the term ${\|\hat\mu_{\mathcal C}^1- \hat\mu_{\mathcal C}\|}_\textsc{tv}$ in~\eqref{eq:main-trig}, for which we use SSM.
	Let $\Omega(A)$ be the set of all possible configurations on the set $A \subseteq V$.
	For a configuration $\psi$ on $U$, let $\hat\mu_{\mathcal C}^\psi$ denote the marginal of $\mu_B^\psi$ on $\mathcal C$.
	Let $\hat\mu_{e_i}^1$, $\hat\mu_{e_i}^\psi$ be the marginals of $\hat\mu_{B_i}^1$, $\hat\mu_{B_i}^\psi$ on $e_i$, respectively. Then,
	\begin{align*}
	{\|\hat\mu_{\mathcal C}^1- \hat\mu_{\mathcal C}\|}_\textsc{tv} 
	&\le \sum_{\psi \in \Omega(U)} \mu(\psi) 	{\|\hat\mu_{\mathcal C}^1- \hat\mu_{\mathcal C}^\psi\|}_\textsc{tv}  
	\le  \sum_{\psi \in \Omega(U)} \sum_{i=1}^N \mu(\psi) {\|\hat\mu_{e_i}^1- \hat\mu_{e_i}^\psi\|}_\textsc{tv} 
	\le \frac{N}{e^{\kappa\ell}} = o(1),
	\end{align*}
	where the second inequality follows from the fact that $\mu_B^1$ and $\mu_B^\psi$ are product measures over the $B_i$'s, and the last one follows from the SSM property for a suitable constant $\kappa > 0$.
	
	It remains for us to provide a lower bound for the term ${\|Y_T(\mathcal C) - \hat\mu_{\mathcal C}^1\|}_\textsc{tv}$ in~\eqref{eq:main-trig}. For this, we introduce an auxiliary copy of the chain $\{Y_t\}$, 
	denoted $\{Z_t\}$, which is coupled with~$\{Y_t\}$ but with a slightly different starting condition.
	Namely, $Z_0$ is sampled proportional to $\mu^1_{B_i}$ on the set $\mathcal A_i$, independently for each $B_i$. 
	(Recall that $Y_0=X_0$ is sampled proportional to $\mu^1_{B_i}$ on $\mathcal S_i \cap \mathcal A_i$ instead.)
	Then,
	\begin{equation}
	\label{eq:second:traig}
	{\|Y_T(\mathcal C) - \hat\mu_{\mathcal C}^1\|}_\textsc{tv} \ge {\|Z_T(\mathcal C) - \hat\mu_{\mathcal C}^1\|}_\textsc{tv} - {\|Y_T(\mathcal C) -Z_T(\mathcal C)\|}_\textsc{tv}.
	\end{equation}
	We first provide an upper bound for the second term in~\eqref{eq:second:traig}.  Plainly,
	$$
	{\|Y_T(\mathcal C) -Z_T(\mathcal C)\|}_\textsc{tv} \le {\|Y_T -Z_T\|}_\textsc{tv} \le  \Pr[Y_0 \neq Z_0].
	$$
	Let $\mu_0^Y$, $\mu_0^Z$ be the initial distribution for $\{Y_t\}$ and $\{Z_t\}$, respectively, and let $\mathcal S = \otimes \mathcal S_i$ and $\mathcal A = \otimes \mathcal A_i$. For $\sigma \in \mathcal S \cap \mathcal A$, we have
	$\mu_0^Y(\sigma) = \mu^1_B(\sigma)/ \mu^1_B(\mathcal S \cap \mathcal A)$, and
	for $\sigma \in \mathcal A$,
	$\mu_0^Z(\sigma) = \mu^1_B(\sigma)/ \mu^1_B(\mathcal A)$.
	Therefore, 
	if the configurations $Y_0$ and $Z_0$ are sampled from the optimal coupling
	between $\mu_0^Y$, $\mu_0^Z$ and the steps of $\{Y_t\}$, $\{Z_t\}$ are then coupled with the SW coupling, we have 
	$$
	{\|Y_T(\mathcal C) -Z_T(\mathcal C)\|}_\textsc{tv} \le {\|\mu_0^Y -\mu_0^Z\|}_\textsc{tv}  \le \sum_{i=1}^N {\|\mu_0^{Y,i} -\mu_0^{Z,i}\|}_\textsc{tv} = N {\|\mu_0^{Y,1} -\mu_0^{Z,1}\|}_\textsc{tv},
	$$
	where $\mu_0^{Y,i}$, $\mu_0^{Z,i}$ are the initial distributions of $Y_0$, $Z_0$ on $B_i$. Then,
	\begin{align*}
	{\|\mu_0^{Y,1} -\mu_0^{Z,1}\|}_\textsc{tv} 
	&= \frac{\mu^1_{B_1}(\mathcal A_1 \setminus \mathcal S_1)}{\mu^1_{B_1}(\mathcal A_1)} \le \frac{\mu^1_{B_1}(\mathcal S_1^c)}{\mu^1_{B_1}(\mathcal A_1)} = O\left( {e^{-c L}}\right),
	\end{align*}
	where the last inequality follows from Lemma~\ref{lemma:sht:stat}
	and the fact that $\mu^1_{B_1}(\mathcal A_1) = \Omega(1)$ . 
	In summary, since $L = C (\log n)$ and $C$ can be taken large enough, we have proved 
	$$
	{\|Y_T(\mathcal C) -Z_T(\mathcal C)\|}_\textsc{tv} = o(1).
	$$
	
	It remains for us to find a lower bound for ${\|Z_T(\mathcal C) - \hat\mu_{\mathcal C}^1\|}_\textsc{tv}$ in~\eqref{eq:second:traig} for a suitable $T$.
	For a configuration $\sigma$ on $B$, let $f(\sigma)$ denote the number of edges $e_i \in \mathcal C$ that are monochromatic in $\sigma$.
	For any $a \ge 0$
	we have 
	\begin{equation}
	\label{eq:lb:3}
	{\|Z_T(\mathcal C) - \hat\mu_{\mathcal C}^1\|}_\textsc{tv} \ge \Pr[f(Z_T) \ge a] - \Pr\nolimits_{\sigma \sim \mu^1_{B}}[f(\sigma) \ge a].
	\end{equation}
	We will show that, for a suitable $T$ and any $i=1,\dots,N$,
	\begin{equation}
	\label{eq:lb:key}
	\Pr[Z_T(B_i) \in \mathcal A_i] \ge \mu^1_{B_i}(\mathcal A_i) + \frac{1}{N^{1/4}}.
	\end{equation}
	Assuming this is the case, then setting $\mathcal W = \sum_{i=1}^N \mu^1_{B_i}(\mathcal A_i)$
	we obtain by Hoeffding's inequality 
	$$
	\Pr\left[f(Z_T) \ge \mathcal W + N^{3/4} - \sqrt{N \log N}\right] \ge 1 - \frac{1}{N^2}
	$$
	and
	$$
	\Pr\nolimits_{\sigma \sim \mu^1_{B}}\left[f(\sigma) \ge \mathcal W + \sqrt{N \log N}\right] \le \frac{1}{N^2},
	$$
	which yields from~\eqref{eq:lb:3} that
	$
	{\|Z_T(\mathcal C) - \hat\mu_{\mathcal C}^1\|}_\textsc{tv} \ge 1 - {2}/{N^2}
	$
	by taking, e.g., $a = \mathcal W + \sqrt{N \log N}$.
	
	To establish~\eqref{eq:lb:key}, note that by Lemma~\ref{lem:HS} 
	\begin{equation}
	\label{eq:lb:kb}
	\Pr(Z_T(B_i)  \in \mathcal A_i) \ge \mu_{B_i}^1(\mathcal A_i) + (1-\mu_{B_i}^1(\mathcal A_i))^{-T+1}(\Pr(Z_1(B_i)\in \mathcal A_i)-\mu_{B_i}^1(\mathcal A_i))^{T}.
	\end{equation}
	We remark that $\{Z_t\}$ has positive semidefinite transition matrix; 
	this follows from
	the fact $\{Z_t\}$ is a product of SW dynamics in each $B_i$, and the SW dynamics has positive semidefinite transition matrix~\cite{BCSV}. 
	
	Let $\PSWI$ denote the transition matrix of the SW dynamics on $B_i$.
	Then
	\begin{align}
	\Pr(Z_1(B_i)\in \mathcal A_i) &= \sum_{\sigma \in \mathcal A_i} \frac{\mu_{B_i}^1(\sigma)}{\mu_{B_i}^1(\mathcal A_i)} \PSWI(\sigma,\mathcal A_i) = \sum_{\sigma \in \mathcal A_i} \frac{\mu_{B_i}^1(\sigma)}{\mu_{B_i}^1(\mathcal A_i)} \left(\theta(\sigma) + \frac{1-\theta(\sigma)}{q}\right) \notag\\
	&= \frac 1q + \frac{q-1}{q \mu_{B_i}^1(\mathcal A_i)} \sum_{\sigma \in \mathcal A_i} \mu_{B_i}^1(\sigma)\theta(\sigma)~\label{eq:lb:eq1},
	\end{align}
	where $\theta(\sigma)$ denotes the probability that, after the edge percolation phase of the SW step, the end points of the edge $e_i$ are connected in the edge configuration. Similarly,
	\begin{align*}
	\mu_{B_i}^1(\mathcal A_i) &= \sum_{\sigma \in \Omega(B_i)} \mu_{B_i}^1(\sigma) \PSWI(\sigma,\mathcal A_i) = \sum_{\sigma \in \Omega(B_i)\setminus\mathcal A_i} \mu_{B_i}^1(\sigma) \PSWI(\sigma,\mathcal A_i) + \sum_{\sigma \in \mathcal A_i} \mu_{B_i}^1(\sigma) \PSWI(\sigma,\mathcal A_i) \notag\\
	&= \sum_{\sigma \in \Omega(B_i)\setminus\mathcal A_i} \frac{\mu_{B_i}^1(\sigma)}{q} +  \sum_{\sigma \in \mathcal A_i} {\mu_{B_i}^1(\sigma)}\left(\theta(\sigma) + \frac{1-\theta(\sigma)}{q}\right)
	= \frac 1q + \frac{q-1}{q}\sum_{\sigma \in \mathcal A_i} \mu_{B_i}^1(\sigma) \theta(\sigma).
	\end{align*}
	Combining with~\eqref{eq:lb:eq1} we get
	\begin{align*}
	\Pr(Z_1(B_i)\in \mathcal A_i) - \mu_{B_i}^1(\mathcal A_i) 
	&= \frac{q-1}{q} \left(\frac{1}{\mu_{B_i}^1(\mathcal A_i)} - 1\right)\sum_{\sigma \in \mathcal A_i} \mu_{B_i}^1(\sigma) \theta(\sigma) \\
	&\ge  \frac{q-1}{q} \left(\frac{1}{\mu_{B_i}^1(\mathcal A_i)} - 1\right) p \cdot \mu_{B_i}^1(\mathcal A_i) 
	= \frac{q-1}{q} \left(1 - \mu_{B_i}^1(\mathcal A_i)\right) p,
	\end{align*}
	where in the last inequality we use the fact that $\theta(\sigma) \ge p$ when $\sigma \in \mathcal A_i$; recall that $ p = 1 - e^{-\beta}$.
	
	Plugging this bound into~\eqref{eq:lb:kb}, we obtain
	\begin{align*}
	\Pr(Z_T(B_i)  \in \mathcal A_i) &\ge \mu_{B_i}^1(\mathcal A_i) + (1-\mu_{B_i}^1(\mathcal A_i))^{-T+1}\left(\frac{q-1}{q} \left(1 - \mu_{B_i}^1(\mathcal A_i)\right) p\right)^{T} \\
	&= \mu_{B_i}^1(\mathcal A_i) + (1-\mu_{B_i}^1(\mathcal A_i))\left(\frac{(q-1)p}{q} \right)^{T}
	\ge \mu_{B_i}^1(\mathcal A_i) + \frac{1}{N^{1/4}},
	\end{align*}
	where the last inequality holds for $T = \xi \log n$ for a suitable constant $\xi > 0$
	since $\mu_{B_i}^1(\mathcal A_i) = \Omega(1)$.
	%
	%
\end{proof}


We provide next the proof of Theorem~\ref{thm:dp}, our bound on the speed of disagreement propagation under the SW coupling.

\begin{proof}[Proof of Theorem~\ref{thm:dp}]
	We will show inductively that with high probability
	disagreements propagate a distance of at most $L$ in each step.
	Let $\Lambda_i(k) \subseteq B_i$ be the cube of side length $k <\ell$ centered at $v_i$;
	recall that $e_i = \{v_i,u_i\}$ where $v_i$ is the center of $B_i$.
	Let $\Lambda(k) = \bigcup_{i=1}^N \Lambda_i(k)$.
	Note that at time $0$, $X_0$ and $Y_0$ agree on 
	$B = \Lambda(\ell)$.
	
	Let us assume that $X_t$ and $Y_t$ agree on $\Lambda(k)$ for some $k \le \ell-2L$. 
	Suppose $Y_t$ is $L$-shattered in each $B_i$; i.e., $Y_t(B_i) \in \mathcal S_i$ for $i=1,\dots,N$.
	If $E(k)$ is the set of edges with both endpoints in~$\Lambda(k)$, 
	after adding the monochromatic edges of $E(k)$ 
	in $X_t$ and $Y_t$ coupled with the SW coupling, the joint edge/spin configuration on $(\Lambda(k),E(k))$ will be the same in both copies. 
	However, when assigning the new spins, the connected components are not necessarily the same since there can be external connections; i.e., monochromatic paths in $V \setminus \Lambda(k)$.
	This may create disagreements between the two chains on $\Lambda(k)$ but only in the components touching the boundary of $\Lambda(k)$. Since we are assuming that $Y_t$ is $L$-shattered in each $B_i$, then with probability $1-N |B_i|\cdot\exp(-\gamma L)$,
	the disagreements cannot propagate to $\Lambda(k-2(L+1))$.
	Consequently, the spin configurations of $X_{t+1}$ and $Y_{t+1}$ on $\Lambda(k-2(L+1))$ are the same.
	
	Proceeding inductively, and assuming that $Y_t$ is $L$-shattered
	in each $B_i$ for all $t = 0, \dots, T$, 
	we deduce from a union bound that
	$X_T$ and $Y_T$ agree on $\Lambda(\ell-2L - 2T(L+1))$
	with probability at least $1-T N |B_i|\exp(-\gamma L)$,
	provided $\ell > 2T(L+1)+2L$.
	Therefore, $X_t(\mathcal C) = Y_t(\mathcal C)$ for all $t \le T$
	since $\mathcal C \subseteq \Lambda(\ell-2L - 2t(L+1))$.
	
	It remains for us to show that $Y_t$ is $L$-shattered
	in each $B_i$ for all $t = 0, \dots, T$ with probability at 
	least $1-o(1)$.	
	The configuration of $Y_0$ on $B_i$ is sampled proportional to $\mu_{B_i}^1$
	on $\mathcal S_i \cap \mathcal A_i$.
	For $\sigma \in \mathcal S_i \cap \mathcal A_i$, let $\pi_i(\sigma) = \mu_{B_i}^1(\sigma)/\mu_{B_i}^1(\mathcal S_i \cap \mathcal A_i)$
	and for $\sigma \in \mathcal S_i$, let $\hat\pi_i(\sigma) = \mu_{B_i}^1(\sigma)/\mu_{B_i}^1(\mathcal S_i)$.
	We have
	\begin{align*}
	{\Pr}_{Y_0(B_i) \sim \hat\pi_i}(Y_t(B_i) \in \mathcal S_i) 
	= &	{\Pr}_{Y_0(B_i) \sim \hat\pi_i}(Y_t(B_i) \in \mathcal S_i \mid Y_0(B_i) \in \mathcal A_i ) {\Pr}_{Y_0(B_i) \sim \hat\pi_i} (Y_0(B_i) \in \mathcal A_i) + \\
	&{\Pr}_{Y_0(B_i) \sim \hat\pi_i}(Y_t(B_i) \in \mathcal S_i \mid Y_0(B_i) \not\in \mathcal A_i ) {\Pr}_{Y_0(B_i) \sim \hat\pi_i} (Y_0(B_i) \not\in \mathcal A_i),
	\end{align*}	
	and so
	\begin{align}		
	{\Pr}_{Y_0(B_i) \sim \pi_i}(Y_t(B_i) \in \mathcal S_i) 
	&={\Pr}_{Y_0(B_i) \sim \hat\pi_i}(Y_t(B_i) \in \mathcal S_i \mid Y_0(B_i) \in \mathcal A_i ) \notag\\
	&\ge 	1 - \frac{1-{\Pr}_{Y_0(B_i) \sim \hat\pi_i}(Y_t(B_i) \in \mathcal S_i)}{{\Pr}_{Y_0(B_i) \sim \hat\pi_i} (Y_0(B_i) \in \mathcal A_i)} \label{eq:at}
	\end{align}	
	By Lemmas~\ref{lem:HS} and~\ref{lemma:sht:stat},
	$$
	{\Pr}_{Y_0(B_i) \sim \hat\pi_i}(Y_t(B_i) \in \mathcal S_i) \ge \mu_{B_i}^1(\mathcal S_i) \ge 1 - \frac{1}{e^{cL}}.
	$$
	Moreover, ${\Pr}_{Y_0(B_i) \sim \hat\pi_i} (Y_0(B_i) \in \mathcal A_i) = \alpha(\beta,q,d)=\Omega(1)$, and so we obtain from~\eqref{eq:at}
	$$
	{\Pr}_{Y_0(B_i) \sim \pi_i}(Y_t(B_i) \in \mathcal S_i) = 1 - O\left(\frac{1}{e^{cL}}\right).
	$$
	Setting $\mathcal S = \otimes_{i=1}^N \mathcal S_i$,
	a union bound over the $B_i$'s implies
	$$
	{\Pr}(Y_t \in \mathcal S) = 1 - O\left(\frac{N}{e^{cL}}\right).
	$$
	It follows from
	another union bound over the steps that
	$
	\Pr(\forall t \le T: Y_t \in \mathcal S) \ge  1 - O\left(\frac{T N }{e^{cL}}\right).
	$	
	Setting $T = A (\log n)$
	(which satisfies $\ell > 2T(L+1)+2L$ as required),
	recalling that $N = \Theta(n/\ell^{d})$, and taking $C$ sufficiently large, we
	obtain that  
	$
	\Pr(\forall t \le T: Y_t \in \mathcal S) \ge  1 - o(1),
	$
	and hence
	$$
	\Pr\left[\forall t \le A \log n:~X_t(\mathcal C) = Y_t(\mathcal C)\right] = 1-o(1),
	$$
	as claimed.		
\end{proof}

\subsection{Proof of auxiliary lemmas}
\label{subsec:aux:lb}

We conclude the section with the proofs of Lemma~\ref{lem:HS} and~\ref{lemma:sht:stat}.

\begin{proof}[Proof of Lemma~\ref{lem:HS}]
	From the spectral decomposition (see, e.g.,~\cite{LP}), one has
	\begin{equation}\label{eq:bo2}
	\Pr(X_t\in B)= \pi(B) + (1-\pi(B))\sum_{\ell=2}^{|\Omega|}\kappa_\ell \lambda_\ell^t,
	\end{equation}
	where $\kappa_\ell\geq 0$ and $\sum_{\ell=2}^{|\Omega|}\kappa_\ell=1$, and $\lambda_2,\dots,\lambda_{|\Omega|}$ denote the non-negative eigenvalues of the transition matrix of the Markov chain except for the principal eigenvalue $\lambda_1=1$. In particular,
	\begin{equation}\label{eq:bo3}
	\Pr(X_1\in B)- \pi(B) = (1-\pi(B))\sum_{\ell=2}^{|\Omega|}\kappa_\ell \lambda_\ell.
	\end{equation}
	The convexity of the function $f(x) = x^t$ for $t\geq 1$, $x \ge 0$ and Jensen's inequality imply that 
	\begin{equation*}\label{eq:bo4}
	\Pr(X_t\in B)\geq \pi(B) + (1-\pi(B))\left(\sum_{\ell=2}^{|\Omega|}\kappa_\ell \lambda_\ell\right)^t.
	\end{equation*}
	From \eqref{eq:bo3} it follows that  $\sum_{\ell=2}^{|\Omega|}\kappa_\ell \lambda_\ell=(1-\pi(B))^{-1}(\Pr(X_1\in B)-\pi(B))$, and so for $t \ge 1$
	$$
	\Pr(X_t\in B)\geq \pi(B) + (1-\pi(B))^{-t+1}\left((\Pr(X_1\in B)-\pi(B))\right)^t.
	$$
	Finally, observe also that from~\eqref{eq:bo2} it follows that 
	$\Pr(X_t\in B)- \pi(B) \ge 0$.
\end{proof}

\begin{proof}[Proof of Lemma~\ref{lemma:sht:stat}]
	Let $G=(V,\EdgeSet)$ be an $n$-vertex cube in $\bbZ^d$.
	Let $\mathcal {E}_L$ be the set of all edge configurations $A \subseteq \EdgeSet$
	such that for all $v \in V$ at distance at least $2L$ from partial $\partial V$,
	the connected component of $v$ in $A$ 
	does not reach the boundary of the cube $\Lambda_v(L)$ centered at $v$ of side length $L$.	
	Consider the admissible boundary condition $(1,1)$ of the joint space that is wired and all spins are $1$; i.e., in Definition~\ref{def:admissible}  we set $V_0 = \partial V$, $E_0 = \partial E$, $\psi = 1$ and $\varphi=1$.
	From Theorem 1.2 in~\cite{DRT}, we get that for a suitable constant $a > 0$
	\begin{equation}
	\label{eq:lb:final}
	\nu^{(1,1)}(\mathcal {E}_L) \ge 1 - \frac{n}{e^{aL}}.
	\end{equation}	
	We remark that Theorem 1.2 from~\cite{DRT} is stated for the random-cluster measure with the wired boundary condition, but our statement for the joint measure $\nu$ is equivalent; see Section~\ref{sec:rc} below  for a definition of the random-cluster measure and its boundary conditions. We also note that Theorem 1.2 from~\cite{DRT} requires a weaker (easier to satisfy) condition than SSM.
	Recall that $\mathfrak S$ is the set of $L$-shattered configurations of $V$. Then,
	\begin{align*}
	\nu^{(1,1)}(\mathcal {E}_L) 
	&= \sum_{\sigma \in \mathfrak S} \mu^1(\sigma) \nu^{(1,1)}(\mathcal {E}_L\mid \sigma) + \sum_{\sigma \in \mathfrak S^c} \mu^1(\sigma) \nu^{(1,1)}(\mathcal {E}_L\mid \sigma) \\
	&\le \mu^1(\mathfrak S) + (1 - \mu^1(\mathfrak S))\left(1 - \frac{n}{e^{\gamma L}}\right) = 1 - \frac{n}{e^{\gamma L}} + \frac{n \mu^1(\mathfrak S)}{e^{\gamma L}}.
	\end{align*}
	Combining this with~\eqref{eq:lb:final}, we obtain that 
	$
	\mu^1(\mathfrak S) \ge 1 - \exp(-(a-\gamma)L);
	$
	the result follows by choosing $\gamma = a/2 = c$.
\end{proof}

\section{Entropy decay for dynamics in the joint space}\label{sec:entropyjoint}

In this section we study the implications of
our spin/edge factorization of entropy with respect to the joint measure $\nu$
for various dynamics on the joint space on $\bbZ^d$.

\subsection{Swendsen-Wang in the joint space}
 First, we consider the SW dynamics in the joint space.
Let $K$ denote the $|\Joint| \times |\Joint|$ stochastic matrix corresponding
to re-sampling the spins of a joint configuration given the edges, and similarly 
let $T$ be the stochastic matrix corresponding to re-sampling the edges given the spins. 
Specifically, 
\begin{align*}
K ((\s,A),(\t,B)) &= \IND(A=B) \nu (\tau \mid A ) \\
T ((\s,A),(\t,B)) &= \IND(\s=\t) \nu (B \mid \s ).
\end{align*}
Note that $T = T^* = T^2$ and $K = K^* = K^2$; i.e., $K$ and $T$ are self-adjoint idempotent operators.

The Markov chains with transition matrices $KT$ and $TK$ are natural variants of the SW dynamics in the joint space.  In the terminology of~\cite{DSC}, they
are the Markov chains in the joint space corresponding to the two-component Gibbs sampler.
The chains with transition matrices $\frac12(K+T)$, $KTK$ and $TKT$  are also of interest as reversible versions of 
$KT$ and $TK$.
We show that, under SSM, all of these dynamics
satisfy entropy decay with respect to~$\nu$ and hence have
$O(\log n)$ mixing time.
\begin{theorem}\label{th:swtheorem}
	Let $P$ be {\em any} of the stochastic matrices 
	$KT$, $TK$, $\frac12(K+T)$, $KTK$ or $TKT$.
	SSM implies that there exists constant $\d>0$ such that, 
	for all functions $f:\Joint \mapsto \bbR_+$,
	\begin{align*}
	\ent_\nu( P f) \leq (1-\d) \ent_\nu(f).
	\end{align*}
	In particular, 
	the Markov chain with transition matrix $P$ satisfies $\tmix(P) =O(\log n)$.
\end{theorem}

%

First we state the following lemma, which is proved later and will be useful in several of our proofs, including that of Theorem~\ref{th:swtheorem}.
	
\begin{lemma}
	\label{lemma:scan}
	Let $S$ and $S'$ be two idempotent stochastic matrices reversible with respect to a distribution
	$\pi$ over $\Gamma$, and let $Q = \frac{S+S'}{2}$. 
	Suppose there exists $\delta \in (0,1)$ such that, for any positive function $f:\Gamma \mapsto \bbR$, we have
	$
	\ent_\pi(Qf) \le (1-\delta)\ent_\pi(f)
	$. Then
	$
	\ent_\pi(SS'f) \le (1-\delta)\ent_\pi(f)
	$
	and
	$
	\ent_\pi(S'Sf) \le (1-\delta)\ent_\pi(f)
	$.
\end{lemma}	

We are now ready to prove Theorem~\ref{th:swtheorem}.

\begin{proof}[Proof of Theorem~\ref{th:swtheorem}]
	Let us consider first the case when $P = \frac{K+T}{2}$.
	Since $P = P^*$, from Lemma~\ref{lem:mix} and Remark~\ref{rmk:entropy-mixing}
	it is sufficient to prove that, for all functions $f: \Joint \mapsto \bbR_+$ with $\mu[f]=1$, 
	\begin{align*}
	\ent_\nu(P f)\leq (1-\d) \ent_\nu( f).
	\end{align*}
    The  convexity of the function $x \log x$ implies 
	\begin{align}\label{convj1}
	Pf\log (Pf)\leq \frac12 Kf\log (K f) + \frac12 Tf\log (T f).
	\end{align}
	If $\nu[f]=1$, then $\nu[Pf]=\nu[Kf]=\nu[Tf]=1$, and therefore taking expectations with respect to $\nu$ in \eqref{convj1} we obtain
	\begin{align}
	\label{convu1}
	\ent_\nu( Pf )\leq \frac12\left[ \ent_\nu(Kf)+\ent_\nu(Tf)\right].
	\end{align}
	Noting that $Kf (\si,A)= \nu(f\tc A)$ and $Tf (\si,A)= \nu(f\tc\si)$,
	the decompositions in~\eqref{eq:total:entEdges} and~\eqref{eq:total:entSpins} imply
	\begin{align*}
	\ent_\nu(f)=\ent_\nu(Kf)+\nu\left[\ent_\nu(f\tc A)\right]= \ent_\nu(Tf)+\nu\left[\ent_\nu(f\tc \si)\right].
	\end{align*}
	Hence, \eqref{convu1} becomes
	\begin{align}\label{convj2}
	\ent_\nu( Pf ) \leq \ent_\nu(f) - \frac12\nu\left[\ent_\nu(f\tc A)+ \ent_\nu(f\tc \si)\right].
	\end{align}
	Lemma \ref{lemma:main-intro-ad} now implies 
	\begin{align*}
	\ent_\nu( Pf)\leq (1-\d)\ent_\nu( f),
	\end{align*}
	with $\d=1/2C$.
	This proves the theorem for the case when $P=\frac{K+T}{2}$.
	 The result for $KT$ and $TK$ follows from Lemma~\ref{lemma:scan}, by noting that
	 $K^2 = K = K^*$ and $T^2 = T = T^*$, and noting that $(KT)^*=TK$ and $(TK)^*=KT$. 
	 (Note that in Lemma~\ref{lemma:scan}, we do not require the matrices to be ergodic.) 
	 Finally, the cases $P=KTK$, $P=TKT$ follow from the cases $P=KT$ and $P=TK$ with the observation that, by \eqref{conv7},
	 $ \ent_\nu(KTK f)\leq \ent_\nu( TK f)$ and $\ent_\nu( TKT f)\leq \ent_\nu( KT f)$.
\end{proof}

Finally, we go back and supply the missing proof of Lemma~\ref{lemma:scan}.

\begin{proof}[Proof of Lemma~\ref{lemma:scan}]
	Let us first show that 
	\begin{align}\label{conv5}
	\ent_\pi (S f)\leq \ent_\pi (\wt S f),
	\end{align}
	where $\wt S = \frac{S+I}{2}$ is a lazy version of $S$; $I$ denotes the identity matrix.
	To this end, define $U_n = \left[\frac12(S+ I)\right]^n$.
	Then we have $U_1=\wt S$ and $U_n = \frac{(2^n-1)S + I}{2^n}\to S$ as $n\to\infty$. Therefore, \eqref{conv5} follows if we prove that for all $n\geq 1$
	\begin{align}\label{conv6}
	\ent_\pi\, (U_{n+1}f)\leq \ent_\pi(U_n f).
	\end{align}
	On the other hand, if $U$ is any stochastic matrix with stationary distribution $\pi$, then for any function $f: \Gamma \mapsto \bbR_+$ with $\pi[f]=1$ we have $\pi[Uf] =1$. Hence, 
	$\ent_\pi (Uf) = \pi[ (Uf)\log(Uf)]$. Since $U$ is a stochastic matrix, the convexity of the function $x\log x$ implies $(Uf)\log(Uf) \le U(f\log f)$, and so
	\begin{align}\label{conv7}
	\ent_\pi (Uf) &\leq \pi[U(f\log f)] = \pi[f\log f]	= \ent_\pi(f).
	\end{align}
	Since $U_{n+1}f=U_1U_n f$, applying \eqref{conv7} with $f$ replaced by $U_nf $ and with $U=U_1$ proves \eqref{conv6} and \eqref{conv5}.
	We note that since $(S')^2=S'$,  $$\wt S S' = \frac12(S+S')S'=QS'.$$
	Applying \eqref{conv5} with $f$ replaced by $S'f$ we obtain
	\begin{align*}
	\ent_\pi (S S'f)\leq  \ent_\pi (\wt S S' f) = \ent_\pi (QS'f) \leq (1-\d)\ent_\pi (S'f) \leq (1-\d)\ent_\pi (f),
	\end{align*}
	where 
	the second inequality follows from the assumption that $Q$ contracts entropy for any function
	and the last one follows again from \eqref{conv7}. 
	This completes the proof for $SS'$. The same argument with $S$ and $S'$ exchanged applies for $S'S$ and we are done.
\end{proof}	

\subsection{The local dynamics in the joint space}
\label{subsec:local}

In this section, we use Lemma~\ref{lemma:main-intro-ad} to derive tight bounds for the local (Glauber)
dynamics in the joint space; this dynamics has been recently considered in~\cite{CL}, 
but as far as we know there are no results in the literature concerning its rate of 
convergence to stationarity. The dynamics is defined as follows:
in each step, with probability $1/2$ update a vertex and with probability $1/2$ update an edge. To update a vertex, pick $v \in V$ uniformly at random and perform a ``heat-bath'' update at $v$ (i.e., replace the spin of $v$ with a new spin sampled from the conditional
distribution of the spin at $v$ given the current spin/edge configuration);
to update an edge,  pick $e\in \bbE$ uniformly at random and perform a ``heat-bath'' update at~$e$.


For any $v \in V$, $e \in \bbE$, let $Q_v$ denote the stochastic matrix corresponding to the single heat-bath update at vertex $v$, and let $W_{e}$ denote the 
stochastic matrix for the single heat-bath update at the edge $e$. Then the transition matrix $P_{\textsc{local}}$ of the Glauber dynamics in the joint space is given by
\begin{align}\label{locdyn}
P_{\textsc{local}} = \frac1{2|V|}\sum_{v\in V} Q_v + \frac1{2|\bbE|}\sum_{e\in \bbE} W_e.
\end{align}

\begin{theorem}\label{cor:loc:intro}
	SSM implies that there exists a constant $\d>0$ such that, 
	for all $f:\Joint \mapsto \bbR_+$
	\begin{align}
	\label{eq:local:ent}
	\ent_\nu(P_{\textsc{local}} f) \leq \left(1-\frac{\delta}{n} \right) 	\ent_\nu(f).
	\end{align}
	Moreover, the mixing time of the local dynamics satisfies $\tmix(P_{\textsc{local}}) =O (n \log n)$.
\end{theorem}
The mixing time bound in this theorem is asymptotically tight. 
This follows from the lower bounds in~\cite{HS} by considering the projection of $P_{\textsc{local}}$ on the spins; see Remark~\ref{rmk:prj}.

The heat-bath updates in the joint space are quite simple. 
For a vertex $v \in V$, the heat-bath update at $v$ assigns a new spin to $v$ chosen u.a.r.\ from ~$\{1,\dots,q\}$, provided $v$ is isolated (i.e., there are no edges incident to $v$ in the edge configuration); otherwise, the spin at $v$ does not change.
On the other hand, the heat-bath update at $e \in \EdgeSet$ updates the state of $e$ only if it is monochromatic in the spin configuration; if this is the case, the new state of~$e$ corresponds to a Bernoulli($p$) random variable. 
We note that $Q_v$ and $W_{e}$ are reversible with respect to $\nu$. 
Moreover, they are projection operators in $L^2(\Joint,\nu)$; that is, $Q_v^2=Q_v=Q_v^*$ and $W^2_{e}=W_{e}=W_{e}^*$. 

\begin{proof}[Proof of Theorem~\ref{cor:loc:intro}]
First note that since $Q_v$ and $W_e$ are reversible with respect to $\nu$, so is $P_{\textsc{local}}$
and by Lemma~\ref{lem:mix} and Remark~\ref{rmk:entropy-mixing} it is sufficient for us to establish that 
\begin{align}\label{jj44}
\ent_\nu(P_{\textsc{local}}f)\leq (1-\d/n)\ent_\nu(f)
\end{align}
for all  functions $f:\Joint \mapsto \bbR_+$ such that $\nu[f]=1$. Here $\d>0$ is a constant independent of $n$ and the admissible boundary condition. 

By the convexity of the function $x\log x$, reasoning as in \eqref{convu1}, we can write
\begin{align*}
\ent_\nu(P_{\textsc{local}}f)\leq \frac1{2|V|}\sum_{v\in V}\ent_\nu(Q_vf)
+\frac1{2|\bbE|}\sum_{e\in \bbE}\ent_\nu(W_ef). 
\end{align*}

Let 
$\sigma_{V\setminus v}$ (resp., $A_{\EdgeSet \setminus e}$) denote the spin (resp., edge) configuration excluding~$v$ (resp., excluding~$e$).
Since $Q_vf(\si,A)=\nu(f\tc\si_{V\setminus\{v\}},A)$ and $W_e f(\si,A)=\nu(f\tc\si,A_{\EdgeSet\setminus{e}})$, 
from the decompositions of entropy in~\eqref{eq:total:entEdges} 
and~\eqref{eq:total:entSpins} we obtain
\begin{align*}
\ent_\nu(Q_vf) &=\ent_\nu(f) - \nu\left[\ent_\nu(f\tc\si_{V\setminus\{v\}},A)\right];\\
\ent_\nu(W_ef)&=\ent_\nu(f)- \nu\left[\ent_\nu(f\tc\si,A_{\bbE\setminus{e}})\right].
\end{align*}
Therefore,
\begin{align*}
\ent_\nu(P_{\textsc{local}}f)\leq \ent_\nu(f) - \frac1{2|V|}\sum_{v\in V}\nu\left[\ent_\nu(f\tc\si_{V\setminus\{v\}},A)\right]
-\frac1{2|\bbE|}\sum_{e\in \bbE}\nu\left[\ent_\nu(f\tc\si,A_{\bbE\setminus{e}})\right].
\end{align*}
%
We show next that there exists a constant $C > 0$
such that
\begin{align}\label{swent106}
\ent_\nu(f)\leq C
\sum_{v\in V}\nu\left[\ent_\nu(f\tc\si_{V\setminus\{v\}},A)\right]
+C\sum_{e\in \bbE}\nu\left[\ent_\nu(f\tc\si,A_{\bbE\setminus{e}})\right].
\end{align}
The desired estimate \eqref{jj44} then follows from the fact that $|\bbE|=O(|V|)=O(n)$. 

To establish~\eqref{swent106}, note that by Theorem~\ref{thm:cp}
we know that SSM implies approximate even/odd factorization.
Then, from Lemma \ref{lem:conc} we know that, for some constant $C_1>0$, 
\begin{align}\label{swent101}
\ent_\nu(f) \leq C_1 \nu\left[\ent_\nu (f\tc \si_E) + \ent_\nu (f\tc \si_O)\right],
\end{align}
where we recall that $E \subset V$ and $O  \subset V$ are the even and odd sub-lattices, respectively.
Since $\nu(\cdot\tc \si_O)=\otimes_{v\in E}\,\nu_v(\cdot\tc \si_O)$ (see~\eqref{eo1}), the standard tensorization of entropy for  product measures (see, e.g.,~\cite{ABM}) implies
\begin{align}\label{swent102}
\ent_\nu (f\tc \si_O)\leq \sum_{v\in E}\nu\left[\ent_v(f\tc\si_O)\tc\si_O\right],
\end{align}
where as before we use $\ent_v(\cdot\tc\si_O)$ for the entropy with respect to $\nu_v(\cdot\tc \si_O)$. 
From Lemma \ref{lem:tensorx} we see that 
\begin{align}\label{swent103}
\ent_\nu (f\tc \si_O)\leq C_1\sum_{v\in E}\nu\left[\ent_v(f\tc\si_O,A)+\ent_v(f\tc\si)\tc\si_O\right],
\end{align}
for some constant $C_1>0$. 

For $v\in E$, the distribution of the spin $\sigma_v$ given $\si_O$ and $A$ is the same as the distribution of $\si_v$ given $\si_{V\setminus\{v\}}$ and $A$; that is, 
$\nu_v(\cdot\tc\si_O,A)= \nu(\cdot\tc\si_{V\setminus\{v\}},A)$. Therefore we may write
\begin{align}\label{swent1a3}
\ent_v(f\tc\si_O,A)= \ent_\nu(f\tc\si_{V\setminus\{v\}},A).
\end{align}
Let us also observe that, for every $v\in E$,
\begin{align}\label{swent104}
\ent_v(f\tc\si)\leq \sum_{w \in V:\,\{w,v\}\in\EdgeSet}\ent_\nu(f\tc\si,A_{\bbE\setminus{\{w,v\}}}).
\end{align}
Indeed, $\nu_v(\cdot\tc \si)$ is a product measure on $A_v=\{A_{vw},\,\{w,v\}\in\EdgeSet\}$, and the entropy appearing on the right hand side above is simply the entropy of $A_{vw}$ once every other spin or edge variable has been fixed. Therefore, \eqref{swent104} is again the standard tensorization statement for product measures. In conclusion, we have shown that
\begin{align}\label{swent105}
\ent_\nu (f\tc \si_O)\leq C_1\sum_{v\in E}\nu\left[\ent_\nu(f\tc\si_{V\setminus\{v\}},A)\tc\si_O\right]
+C_1\sum_{e\in \bbE}\nu\left[\ent_\nu(f\tc\si,A_{\bbE\setminus{e}})\tc\si_O\right],
\end{align}
where the second sum is now over the set of all edges~$\bbE$. The same estimate can be obtained with the role of even and odd sites reversed:
\begin{align}\label{swent1050}
\ent_\nu (f\tc \si_E)\leq C_1\sum_{v\in O}\nu\left[\ent_\nu(f\tc\si_{V\setminus\{v\}},A)\tc\si_E\right]
+C_1\sum_{e\in \bbE}\nu\left[\ent_\nu(f\tc\si,A_{\bbE\setminus{e}})\tc\si_E\right].
\end{align}
Taking expectations with respect to $\nu$
and summing~\eqref{swent105} and~\eqref{swent1050}, 
from \eqref{swent101} we obtain~\eqref{swent106} which finishes the proof.
\end{proof}

\begin{remark}
	\label{rmk:prj}
By taking $f$ that depends only on spins, we derive as a corollary of Theorem~\ref{cor:loc:intro} entropy decay for the Potts model Glauber dynamics (up to a constant laziness factor to account for the probability of a site being isolated); similarly, taking $f$ that depends only on edges, we obtain entropy decay for the corresponding Glauber dynamics for the random-cluster model. 
While entropy decay was previously known for the Potts Glauber dynamics under SSM~\cite{Cesi}, the same statement for the random-cluster dynamics appears to be a new result.  (Note in particular that entropy decay does not follow from the mixing time results for this dynamics in~\cite{BSz2}.)
\end{remark}

We briefly mention several other consequences of our results.
First, we note that Theorem~\ref{cor:loc:intro} can be extended to the more general case of (weighted) block dynamics for the joint space. In addition, since the ``edge marginal'' of the joint measure $\nu$ is the random-cluster distribution, we can show that the mixing time of the SW dynamics for the random-cluster model, which alternates between edge and joint configurations, is also $O(\log n)$ for all integer $q \ge 2$ provided SSM holds; see Section~\ref{sec:rc} for more details about our results for random-cluster dynamics. 

Finally, we note that while our results in the joint space are all stated for the free boundary
condition, they actually extend to the more general class of admissible boundary
conditions; see Definition~\ref{def:admissible} in 
Section~\ref{sec:factorization} for the definition of this class.

\section{Entropy decay for the alternating scan dynamics}\label{sec:altscan}

The fact that classical log-Sobolev inequalities do not capture the mixing time of the SW dynamics seems to be a more general phenomenon afflicting non-local Markov chains. These chains are 
popular due to their presumed speed-up over Glauber dynamics and to the fact that their updates
can be parallelized.  With our techniques, we are able to establish entropy contraction for
another standard non-local Markov chain for the Potts model known as the \textit{alternating scan dynamics}. This chain, which
is used in practice to sample from the Gibbs distribution 
and has received some theoretical attention \cite{BCSV,PW,GKZ},
also has a ``bad'' log-Sobolev constant, but we can show that entropy decays at a constant rate over the steps of the chain.

In one step of the alternating scan dynamics, all the {\it even\/} vertices  
(i.e., those with even coordinate sum) are updated simultaneously 
with a new configuration distributed according to the conditional measure on the
even sub-lattice given the configuration on the odd sub-lattice; the process is then repeated for the odd vertices. The key observation is that the conditional distributions on the even and odd sub-lattices are product distributions, which makes this chain particularly amenable to parallelization and thus attractive in applications.

Let $P_E$ be the stochastic matrix corresponding to the update of the even sites conditional on the spins of the odd sites, and define $P_O$ analogously for the odd sites.
The alternating scan dynamics is the Markov chain with transition matrix $S_{EO}=P_EP_O$
(or, equivalently, $S_{OE}=P_OP_E$).  Note that $P_E,P_O$ do not commute  
(unless $\b=0$), so $S_{EO}$ and $S_{OE}$ are not reversible with respect to their stationary measure~$\mu$.
In~\cite{BCSV} it was shown that whenever SSM holds, the mixing time of the reversibilized
version $S_{OEO}=P_OP_EP_O$ of this dynamics is~$O(n)$.  Here we prove a much
tighter bound by showing that the alternating scan dynamics itself contracts entropy at a constant rate.

\begin{theorem}\label{th:alt-intro}
	Let $P$ be either of the stochastic matrices $S_{EO}$ or $S_{OE}$. 
	SSM implies that there exists a constant $\d>0$ such that,
	for
	all boundary conditions and  
	all functions $f:\Omega \mapsto \bbR_+$, 
	\begin{align*}
	\ent_\mu( P f) \leq (1-\d) \ent_\mu(f).
	\end{align*}
	In particular, 
	the Markov chain with transition matrix $P$ satisfies $\tmix(P) =O(\log n)$.
\end{theorem}

We note that
the alternating scan dynamics is a version of so-called {\it systematic
	scan\/} dynamics, a variant of Glauber dynamics in which vertices are updated in some
fixed, rather than random, ordering.  Due to their widespread use in practice,
the effect of decay of correlations properties on the speed of convergence of this class
of dynamics has been widely studied; see, e.g.~\cite{DGJ-dob,Hayes,DGJ-norm}. 
Recently in~\cite{CP}, a result analogous to Theorem~\ref{th:alt-intro} was obtained
for the simpler reversible dynamics with transition matrix $\frac{P_E+P_O}{2}$.

%

\begin{proof}[Proof of Theorem~\ref{th:alt-intro}]
	
	We will show that the discrete entropy contraction in \eqref{relent3} holds for $S_{EO}$ and $S_{OE}$ for any positive function $f:\Omega \mapsto \bbR$ such that $\mu[f]=1$.  
	The mixing time bounds then follow from Lemma \ref{lem:mix}
	and the fact that $S_{EO}^*=S_{OE}$ and $S_{OE}^*=S_{EO}$.
	In view of Lemma~\ref{lemma:scan}, it is sufficient for us to establish \eqref{relent3} for $P:=\frac{P_E+P_O}{2}$. 	
	The convexity of the function $x\log x$ implies the pointwise bound
	\begin{align*}
	(Pf)\log (Pf)\leq \frac12 (P_Ef)\log (P_E f) + \frac12 (P_Of)\log (P_O f).
	\end{align*}
	From this and the fact that $\mu[Pf] = 1$ we get
	\begin{align}\label{conve1}
	\ent_\mu (Pf) = \mu[(Pf) \log (Pf)] \leq \frac12\left[ \ent_\mu(P_Ef)+\ent_\mu(P_Of)\right].
	\end{align}
	Note that $P_E$ and $P_O$ are the orthogonal projections in $L^2(\Omega,\mu)$ such that $P_E f= \mu(f\tc\si_O)$ and $P_O f = \mu(f\tc\si_E)$. Therefore,	
	\begin{align*}
	\ent_\mu(f) &= \ent_\mu (\mu(f\tc\si_O))+\mu\left[\ent_\mu(f\tc \si_O)\right] =  \ent_\mu (P_E f)+\mu\left[\ent_\mu(f\tc \si_O)\right]; \\
	\ent_\mu(f) &= \ent_\mu (\mu(f\tc\si_E))+\mu\left[\ent_\mu(f\tc \si_E)\right] = \ent_\mu (P_Of)+\mu\left[\ent_\mu(f\tc \si_E)\right],
	\end{align*}
	and we see that \eqref{conve1} is equivalent to 
	\begin{align*}
	\ent_\mu(Pf) \leq \ent_\mu(f) - \frac12\mu\left[\ent_\mu(f\tc \si_O)+ \ent_\mu(f\tc \si_E)\right].
	\end{align*}
	We may now apply Theorem~\ref{thm:cp} which implies that, when SSM holds,
	\begin{align}\label{conv3}
	\ent_\mu (Pf) \leq (1-\d)\ent_\mu(f) ,
	\end{align}
	for a suitable constant $\delta \in (0,1)$. 	This establishes~\eqref{relent3} for $P = \frac{P_E+P_O}{2}$.
	Since $P_E^2 = P_E = P_E^*$ and $P_O^2 = P_O = P_O^*$,
	and $(P_EP_O)^*=P_OP_E$, $(P_OP_E)^*=P_EP_O$, 
	the remainder of the result follows from Lemma~\ref{lemma:scan}.
\end{proof}

\section{Random-cluster dynamics}
\label{sec:rc}

In this section we study the implications of
our results for the dynamics of the 
\emph{random-cluster model} for both the high and low temperatures regimes.
This allows us to derive Theorem~\ref{cor:sw-lt:intro} from the introduction using a comparison mechanism we establish in~Section~\ref{subsec:comparison}.

The random-cluster model on $G = (V,\EdgeSet)$ with parameters $p \in (0, 1)$ and $q > 0$ assigns to each $A \subseteq \EdgeSet$ a probability
\begin{equation}\label{eq:rc}
\r(A) = \r_{G,p,q}(A)=\frac1{Z_{\textsc{rc}}}\,p^{|A|}(1-p)^{|\bbE|-|A|}q^{c(A)},
\end{equation}
where $c(A)$ is the number of connected components in $(V,A)$
and $Z_{\textsc{rc}}$ is the corresponding partition function.
The random-cluster model was first introduced by Fortuin and Kasteleyn~\cite{FK} as a unifying framework for random graphs, spin systems  and electrical networks; see the book~\cite{Grimmett} for extensive background.

A {\it boundary condition\/}
for the random-cluster model is a partition $\xi = \{\xi_1,\xi_2,\dots\}$ 
of the internal boundary $\partial V$ of $V$ such that all vertices in each $\xi_i$ are 
constrained to be in the same connected component of any configuration~$A$.
(We can think of the vertices in $\xi_i$ as being connected through a configuration in $V^c$.)
These connections are considered in the counting of the connected components in~\eqref{eq:rc};  i.e., $c(A)$ becomes $c(A,\xi)$ (see, e.g.,~\cite{BGV,Grimmett}).

The distribution $\r$ with a \emph{free} boundary condition (i.e., every element of $\xi$ is a single vertex) corresponds to the edge marginal of the joint measure also with free boundary condition~\eqref{jointnu-intro}; that is, 
$\r(A)=\sum_{\si: A  \subseteq M(\s)} \nu(\si,A)$
and $Z_{\textsc{rc}} = Z_{\textsc{j}}$; see, e.g.,~\cite{ES,Grimmett}. 
The \emph{wired} boundary condition 
corresponds to the case when all vertices of $\partial V$ are connected by the boundary condition (i.e., $\xi = \{\partial V\}$).
More generally, if $(\psi,\phi)$ is an admissible boundary condition for the joint space (see Definition~\ref{def:admissible}), 
we have
\begin{equation}\label{rcsi}
\r^{\psi,\phi}(A)= \sum_{\si: A  \subseteq M(\s)} \nu^{\psi,\phi}(\si,A) = 
\frac1{Z^{\psi,\phi}}\,p^{|A|}(1-p)^{|\bbE|-|A|}q^{c(A)-c_0(A)} \IND(A\sim\psi)\IND(A\sim\phi),
\end{equation}
where
$A\sim\psi$ means that $A$ does not connect vertices of $V_0$ with different colors in $\psi$,
$A \sim \phi$ that $A$ and $\phi$ agree on the edges in $\bbE_0$ and
$c_0(A)$ denotes the number of connected components that intersect $V_0 \subseteq \partial V$; see Figure~\ref{fig7.1} for some admissible boundary conditions.

\begin{figure}
	\center
	
	\begin{subfigure}                           
		
		\begin{tikzpicture}[scale=2]
		

		\draw  (2/4,2/4) -- (2/4,6/4);
		\draw  (3/4,2/4) -- (3/4,6/4);
		\draw  (4/4,2/4) -- (4/4,6/4);
		\draw  (5/4,2/4) -- (5/4,6/4);
		\draw  (6/4,2/4) -- (6/4,6/4);
		\draw  (2/4,2/4) -- (6/4,2/4);
		\draw  (2/4,3/4) -- (6/4,3/4);
		\draw  (2/4,4/4) -- (6/4,4/4);
		\draw  (2/4,5/4) -- (6/4,5/4);
		\draw  (2/4,6/4) -- (6/4,6/4);

		
		

		
		
		
		\draw[fill=red] (1/4,1/4) circle [radius=0.03];
		\draw[fill=red] (2/4,1/4) circle [radius=0.03];
		\draw[fill=red] (3/4,1/4) circle [radius=0.03];
		\draw[fill=red] (4/4,1/4) circle [radius=0.03];
		\draw[fill=red] (5/4,1/4) circle [radius=0.03];
		\draw[fill=red] (6/4,1/4) circle [radius=0.03];
		\draw[fill=red] (7/4,1/4) circle [radius=0.03];
		
		\draw[fill=red] (1/4,1/2) circle [radius=0.03];
		\draw[fill] (2/4,1/2) circle [radius=0.03];
		\draw[fill] (3/4,1/2) circle [radius=0.03];
		\draw[fill] (4/4,1/2) circle [radius=0.03];
		\draw[fill] (5/4,1/2) circle [radius=0.03];
		\draw[fill] (6/4,1/2) circle [radius=0.03];
		\draw[fill=red] (7/4,1/2) circle [radius=0.03];
		
		\draw[fill=red] (1/4,3/4) circle [radius=0.03];
		\draw[fill] (2/4,3/4) circle [radius=0.03];
		\draw[fill] (3/4,3/4) circle [radius=0.03];
		\draw[fill] (4/4,3/4) circle [radius=0.03];
		\draw[fill] (5/4,3/4) circle [radius=0.03];
		\draw[fill] (6/4,3/4) circle [radius=0.03];
		\draw[fill=red] (7/4,3/4) circle [radius=0.03];
		
		\draw[fill=red] (1/4,1) circle [radius=0.03];
		\draw[fill] (2/4,1) circle [radius=0.03];
		\draw[fill] (3/4,1) circle [radius=0.03];
		\draw[fill] (4/4,1) circle [radius=0.03];
		\draw[fill] (5/4,1) circle [radius=0.03];
		\draw[fill] (6/4,1) circle [radius=0.03];
		\draw[fill=red] (7/4,1) circle [radius=0.03];
		
		\draw[fill=red] (1/4,5/4) circle [radius=0.03];
		\draw[fill] (2/4,5/4) circle [radius=0.03];
		\draw[fill] (3/4,5/4) circle [radius=0.03];
		\draw[fill] (4/4,5/4) circle [radius=0.03];
		\draw[fill] (5/4,5/4) circle [radius=0.03];
		\draw[fill] (6/4,5/4) circle [radius=0.03];
		\draw[fill=red] (7/4,5/4) circle [radius=0.03];
		
		\draw[fill=red] (1/4,6/4) circle [radius=0.03];
		\draw[fill] (2/4,6/4) circle [radius=0.03];
		\draw[fill] (3/4,6/4) circle [radius=0.03];
		\draw[fill] (4/4,6/4) circle [radius=0.03];
		\draw[fill] (5/4,6/4) circle [radius=0.03];
		\draw[fill] (6/4,6/4) circle [radius=0.03];
		\draw[fill=red] (7/4,6/4) circle [radius=0.03];
		
		\draw[fill=red] (1/4,7/4) circle [radius=0.03];
		\draw[fill=red] (2/4,7/4) circle [radius=0.03];
		\draw[fill=red] (3/4,7/4) circle [radius=0.03];
		\draw[fill=red] (4/4,7/4) circle [radius=0.03];
		\draw[fill=red] (5/4,7/4) circle [radius=0.03];
		\draw[fill=red] (6/4,7/4) circle [radius=0.03];
		\draw[fill=red] (7/4,7/4) circle [radius=0.03];
		
		\node at (1,0) {\small{(a)}};
		
		\end{tikzpicture}
	\end{subfigure}
	\hfill                                        
	\begin{subfigure}                           
		
		\begin{tikzpicture}[scale=2]
		

		\draw  (1/4,1/4) -- (1/4,7/4);
		\draw  (2/4,1/4) -- (2/4,7/4);
		\draw  (3/4,1/4) -- (3/4,7/4);
		\draw  (4/4,1/4) -- (4/4,7/4);
		\draw  (5/4,1/4) -- (5/4,7/4);
		\draw  (6/4,1/4) -- (6/4,7/4);
		\draw  (7/4,1/4) -- (7/4,7/4);
		\draw  (1/4,1/4) -- (7/4,1/4);
		\draw  (1/4,2/4) -- (7/4,2/4);
		\draw  (1/4,3/4) -- (7/4,3/4);
		\draw  (1/4,4/4) -- (7/4,4/4);
		\draw  (1/4,5/4) -- (7/4,5/4);
		\draw  (1/4,6/4) -- (7/4,6/4);
		\draw  (1/4,7/4) -- (7/4,7/4);

		
		

		
		
		
		\draw[fill=red] (1/4,1/4) circle [radius=0.03];
		\draw[fill=red] (2/4,1/4) circle [radius=0.03];
		\draw[fill=red] (3/4,1/4) circle [radius=0.03];
		\draw[fill=red] (4/4,1/4) circle [radius=0.03];
		\draw[fill=red] (5/4,1/4) circle [radius=0.03];
		\draw[fill=red] (6/4,1/4) circle [radius=0.03];
		\draw[fill=red] (7/4,1/4) circle [radius=0.03];
		
		\draw[fill=red] (1/4,1/2) circle [radius=0.03];
		\draw[fill] (2/4,1/2) circle [radius=0.03];
		\draw[fill] (3/4,1/2) circle [radius=0.03];
		\draw[fill] (4/4,1/2) circle [radius=0.03];
		\draw[fill] (5/4,1/2) circle [radius=0.03];
		\draw[fill] (6/4,1/2) circle [radius=0.03];
		\draw[fill=red] (7/4,1/2) circle [radius=0.03];
		
		\draw[fill=red] (1/4,3/4) circle [radius=0.03];
		\draw[fill] (2/4,3/4) circle [radius=0.03];
		\draw[fill] (3/4,3/4) circle [radius=0.03];
		\draw[fill] (4/4,3/4) circle [radius=0.03];
		\draw[fill] (5/4,3/4) circle [radius=0.03];
		\draw[fill] (6/4,3/4) circle [radius=0.03];
		\draw[fill=red] (7/4,3/4) circle [radius=0.03];
		
		\draw[fill=red] (1/4,1) circle [radius=0.03];
		\draw[fill] (2/4,1) circle [radius=0.03];
		\draw[fill] (3/4,1) circle [radius=0.03];
		\draw[fill] (4/4,1) circle [radius=0.03];
		\draw[fill] (5/4,1) circle [radius=0.03];
		\draw[fill] (6/4,1) circle [radius=0.03];
		\draw[fill=red] (7/4,1) circle [radius=0.03];
		
		\draw[fill=red] (1/4,5/4) circle [radius=0.03];
		\draw[fill] (2/4,5/4) circle [radius=0.03];
		\draw[fill] (3/4,5/4) circle [radius=0.03];
		\draw[fill] (4/4,5/4) circle [radius=0.03];
		\draw[fill] (5/4,5/4) circle [radius=0.03];
		\draw[fill] (6/4,5/4) circle [radius=0.03];
		\draw[fill=red] (7/4,5/4) circle [radius=0.03];
		
		\draw[fill=red] (1/4,6/4) circle [radius=0.03];
		\draw[fill] (2/4,6/4) circle [radius=0.03];
		\draw[fill] (3/4,6/4) circle [radius=0.03];
		\draw[fill] (4/4,6/4) circle [radius=0.03];
		\draw[fill] (5/4,6/4) circle [radius=0.03];
		\draw[fill] (6/4,6/4) circle [radius=0.03];
		\draw[fill=red] (7/4,6/4) circle [radius=0.03];
		
		\draw[fill=red] (1/4,7/4) circle [radius=0.03];
		\draw[fill=red] (2/4,7/4) circle [radius=0.03];
		\draw[fill=red] (3/4,7/4) circle [radius=0.03];
		\draw[fill=red] (4/4,7/4) circle [radius=0.03];
		\draw[fill=red] (5/4,7/4) circle [radius=0.03];
		\draw[fill=red] (6/4,7/4) circle [radius=0.03];
		\draw[fill=red] (7/4,7/4) circle [radius=0.03];
		
			\node at (1,0) {\small{(b)}};
		
		\end{tikzpicture}
	\end{subfigure}
	\hfill                                                 
	\begin{subfigure}                           
		
		\begin{tikzpicture}[scale=2]
		
		
		

		
		
		
		\draw[very thick, blue] (1/4,1/4) --(1/4,7/4); 
		\draw[very thick, blue] (2/4,1/4)--(2/4,2/4);\draw[very thick, blue] (2/4,6/4)--(2/4,7/4);  
		\draw[very thick, blue] (3/4,1/4)--(3/4,2/4); \draw[very thick, blue] (3/4,6/4)--(3/4,7/4); 
		\draw[very thick, blue] (4/4,1/4)--(4/4,2/4); \draw[very thick, blue] (4/4,6/4)--(4/4,7/4); 
		\draw[very thick, blue] (5/4,1/4)--(5/4,2/4); \draw[very thick, blue] (5/4,6/4)--(5/4,7/4); 
		\draw[very thick, blue] (6/4,1/4)--(6/4,2/4); \draw[very thick, blue] (6/4,6/4)--(6/4,7/4); 
		\draw[very thick, blue] (7/4,1/4)--(7/4,7/4);

		\draw[very thick, blue](1/4,1/4)--(7/4,1/4);
		\draw[very thick, blue](1/4,2/4)--(2/4,2/4);\draw[very thick, blue](6/4,2/4)--(7/4,2/4);
		\draw[very thick, blue](1/4,3/4)--(2/4,3/4);\draw[very thick, blue](6/4,3/4)--(7/4,3/4);
		\draw[very thick, blue](1/4,4/4)--(2/4,4/4);\draw[very thick, blue](6/4,4/4)--(7/4,4/4);
		\draw[very thick, blue](1/4,5/4)--(2/4,5/4);\draw[very thick, blue](6/4,5/4)--(7/4,5/4);
		\draw[very thick, blue](1/4,6/4)--(2/4,6/4);\draw[very thick, blue](6/4,6/4)--(7/4,6/4);
		\draw[very thick, blue](1/4,7/4)--(2/4,7/4);\draw[very thick, blue](6/4,7/4)--(7/4,7/4);
		\draw[very thick, blue](1/4,7/4)--(7/4,7/4);

		
		\draw  (2/4,2/4) -- (2/4,7/4);
		\draw  (3/4,2/4) -- (3/4,7/4);
		\draw  (4/4,2/4) -- (4/4,7/4);
		\draw  (5/4,2/4) -- (5/4,7/4);
		\draw  (6/4,2/4) -- (6/4,7/4);
		\draw  (2/4,2/4) -- (6/4,2/4);
		\draw  (2/4,3/4) -- (6/4,3/4);
		\draw  (2/4,4/4) -- (6/4,4/4);
		\draw  (2/4,5/4) -- (6/4,5/4);
		\draw  (2/4,6/4) -- (6/4,6/4);

		
		\draw[fill=red] (1/4,1/4) circle [radius=0.03];
		\draw[fill=red] (2/4,1/4) circle [radius=0.03];
		\draw[fill=red] (3/4,1/4) circle [radius=0.03];
		\draw[fill=red] (4/4,1/4) circle [radius=0.03];
		\draw[fill=red] (5/4,1/4) circle [radius=0.03];
		\draw[fill=red] (6/4,1/4) circle [radius=0.03];
		\draw[fill=red] (7/4,1/4) circle [radius=0.03];
		
		\draw[fill=red] (1/4,1/2) circle [radius=0.03];
		\draw[fill] (2/4,1/2) circle [radius=0.03];
		\draw[fill] (3/4,1/2) circle [radius=0.03];
		\draw[fill] (4/4,1/2) circle [radius=0.03];
		\draw[fill] (5/4,1/2) circle [radius=0.03];
		\draw[fill] (6/4,1/2) circle [radius=0.03];
		\draw[fill=red] (7/4,1/2) circle [radius=0.03];
		
		\draw[fill=red] (1/4,3/4) circle [radius=0.03];
		\draw[fill] (2/4,3/4) circle [radius=0.03];
		\draw[fill] (3/4,3/4) circle [radius=0.03];
		\draw[fill] (4/4,3/4) circle [radius=0.03];
		\draw[fill] (5/4,3/4) circle [radius=0.03];
		\draw[fill] (6/4,3/4) circle [radius=0.03];
		\draw[fill=red] (7/4,3/4) circle [radius=0.03];
		
		\draw[fill=red] (1/4,1) circle [radius=0.03];
		\draw[fill] (2/4,1) circle [radius=0.03];
		\draw[fill] (3/4,1) circle [radius=0.03];
		\draw[fill] (4/4,1) circle [radius=0.03];
		\draw[fill] (5/4,1) circle [radius=0.03];
		\draw[fill] (6/4,1) circle [radius=0.03];
		\draw[fill=red] (7/4,1) circle [radius=0.03];
		
		\draw[fill=red] (1/4,5/4) circle [radius=0.03];
		\draw[fill] (2/4,5/4) circle [radius=0.03];
		\draw[fill] (3/4,5/4) circle [radius=0.03];
		\draw[fill] (4/4,5/4) circle [radius=0.03];
		\draw[fill] (5/4,5/4) circle [radius=0.03];
		\draw[fill] (6/4,5/4) circle [radius=0.03];
		\draw[fill=red] (7/4,5/4) circle [radius=0.03];
		
		\draw[fill=red] (1/4,6/4) circle [radius=0.03];
		\draw[fill] (2/4,6/4) circle [radius=0.03];
		\draw[fill] (3/4,6/4) circle [radius=0.03];
		\draw[fill] (4/4,6/4) circle [radius=0.03];
		\draw[fill] (5/4,6/4) circle [radius=0.03];
		\draw[fill] (6/4,6/4) circle [radius=0.03];
		\draw[fill=red] (7/4,6/4) circle [radius=0.03];
		
		\draw[fill=red] (1/4,7/4) circle [radius=0.03];
		\draw[fill=red] (2/4,7/4) circle [radius=0.03];
		\draw[fill=red] (3/4,7/4) circle [radius=0.03];
		\draw[fill=red] (4/4,7/4) circle [radius=0.03];
		\draw[fill=red] (5/4,7/4) circle [radius=0.03];
		\draw[fill=red] (6/4,7/4) circle [radius=0.03];
		\draw[fill=red] (7/4,7/4) circle [radius=0.03];

			\node at (1,0) {\small{(c)}};

		\end{tikzpicture}
	\end{subfigure}
	\hfill
	\begin{subfigure}                                   
		
		\begin{tikzpicture}[scale=2]
		
		
		\draw[very thick, blue] (1/4,1/4) -- (1/4,7/4); 
		\draw[very thick, blue] (7/4,1/4)--(7/4,7/4); 
		
		
		\draw[very thick, blue](1/4,1/4)--(7/4,1/4);
		\draw[very thick, blue](1/4,7/4)--(7/4,7/4);
		
		
		\draw  (2/4,1/4) -- (2/4,7/4);
		\draw  (3/4,1/4) -- (3/4,7/4);
		\draw  (4/4,1/4) -- (4/4,7/4);
		\draw  (5/4,1/4) -- (5/4,7/4);
		\draw  (6/4,1/4) -- (6/4,7/4);
		\draw  (1/4,2/4) -- (7/4,2/4);
		\draw  (1/4,3/4) -- (7/4,3/4);
		\draw  (1/4,4/4) -- (7/4,4/4);
		\draw  (1/4,5/4) -- (7/4,5/4);
		\draw  (1/4,6/4) -- (7/4,6/4);
		\draw  (1/4,7/4) -- (7/4,7/4);

		
		\draw[fill=red] (1/4,1/4) circle [radius=0.03];
		\draw[fill=red] (2/4,1/4) circle [radius=0.03];
		\draw[fill=red] (3/4,1/4) circle [radius=0.03];
		\draw[fill=red] (4/4,1/4) circle [radius=0.03];
		\draw[fill=red] (5/4,1/4) circle [radius=0.03];
		\draw[fill=red] (6/4,1/4) circle [radius=0.03];
		\draw[fill=red] (7/4,1/4) circle [radius=0.03];
		
		\draw[fill=red] (1/4,1/2) circle [radius=0.03];
		\draw[fill] (2/4,1/2) circle [radius=0.03];
		\draw[fill] (3/4,1/2) circle [radius=0.03];
		\draw[fill] (4/4,1/2) circle [radius=0.03];
		\draw[fill] (5/4,1/2) circle [radius=0.03];
		\draw[fill] (6/4,1/2) circle [radius=0.03];
		\draw[fill=red] (7/4,1/2) circle [radius=0.03];
		
		\draw[fill=red] (1/4,3/4) circle [radius=0.03];
		\draw[fill] (2/4,3/4) circle [radius=0.03];
		\draw[fill] (3/4,3/4) circle [radius=0.03];
		\draw[fill] (4/4,3/4) circle [radius=0.03];
		\draw[fill] (5/4,3/4) circle [radius=0.03];
		\draw[fill] (6/4,3/4) circle [radius=0.03];
		\draw[fill=red] (7/4,3/4) circle [radius=0.03];
		
		\draw[fill=red] (1/4,1) circle [radius=0.03];
		\draw[fill] (2/4,1) circle [radius=0.03];
		\draw[fill] (3/4,1) circle [radius=0.03];
		\draw[fill] (4/4,1) circle [radius=0.03];
		\draw[fill] (5/4,1) circle [radius=0.03];
		\draw[fill] (6/4,1) circle [radius=0.03];
		\draw[fill=red] (7/4,1) circle [radius=0.03];
		
		\draw[fill=red] (1/4,5/4) circle [radius=0.03];
		\draw[fill] (2/4,5/4) circle [radius=0.03];
		\draw[fill] (3/4,5/4) circle [radius=0.03];
		\draw[fill] (4/4,5/4) circle [radius=0.03];
		\draw[fill] (5/4,5/4) circle [radius=0.03];
		\draw[fill] (6/4,5/4) circle [radius=0.03];
		\draw[fill=red] (7/4,5/4) circle [radius=0.03];
		
		\draw[fill=red] (1/4,6/4) circle [radius=0.03];
		\draw[fill] (2/4,6/4) circle [radius=0.03];
		\draw[fill] (3/4,6/4) circle [radius=0.03];
		\draw[fill] (4/4,6/4) circle [radius=0.03];
		\draw[fill] (5/4,6/4) circle [radius=0.03];
		\draw[fill] (6/4,6/4) circle [radius=0.03];
		\draw[fill=red] (7/4,6/4) circle [radius=0.03];
		
		\draw[fill=red] (1/4,7/4) circle [radius=0.03];
		\draw[fill=red] (2/4,7/4) circle [radius=0.03];
		\draw[fill=red] (3/4,7/4) circle [radius=0.03];
		\draw[fill=red] (4/4,7/4) circle [radius=0.03];
		\draw[fill=red] (5/4,7/4) circle [radius=0.03];
		\draw[fill=red] (6/4,7/4) circle [radius=0.03];
		\draw[fill=red] (7/4,7/4) circle [radius=0.03];
			\node at (1,0) {\small{(d)}};
		\end{tikzpicture}
	
	\end{subfigure}
	
	\caption{\footnotesize{The figures above show four distinct admissible boundary conditions of a square region $V$ of the joint space. The boundary condition in (a)
			is obtained by taking 
			$V_0 = \partial{V}$, 
			$\mathbb{E}_0=\partial{\EdgeSet}$, $\psi = \mathrm{``red"}$ and $\phi=0$.
			The boundary condition in (b) is the spin-only monochromatic boundary condition obtained by taking 
			$V_0 = \partial{V},$ 
			$\mathbb{E}_0=\emptyset$ and $\psi = \mathrm{``red"}$; (c) is obtained by taking $V_0=\partial{V}, \mathbb{E}_0=\partial{\mathbb{E}}$, $\psi = \mathrm{``red"}$ and $\phi=1$ (wired edges are colored blue); note that 
			the vertices incident to $\partial V$ will be $\mathrm{``red"}$ with probability $1$.  
			Boundary condition (d) is obtained by taking $V_0=\partial{V}$, $\mathbb{E}_0 = \partial \bbE\setminus \bbE_1$, $\psi = \mathrm{``red"}$ and $\phi=1.$ 
			The marginal on edges of $\nu^{(\psi,\phi)}$ in (a) is the random-cluster measure on the internal square $V\setminus \partial V$
			with the free boundary condition, 
			while in (b), (c) and (d)
			the edge marginal 
			is a wired random-cluster measure over
			$(V, \EdgeSet )$, $(V\setminus \partial{V}, \mathbb{E}\setminus\partial{\mathbb{E}})$
			and $(V, \EdgeSet \setminus(\partial \bbE\setminus \bbE_1))$, respectively. 
		}
	}
	
	\label{fig7.1}
	
\end{figure}
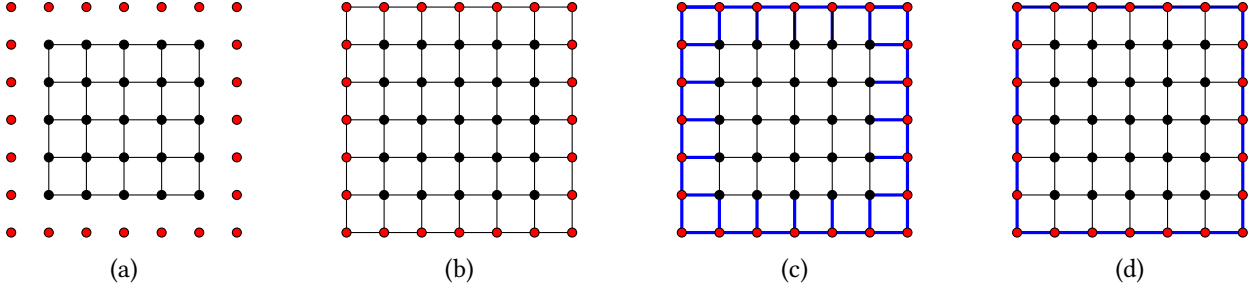

As an example, consider the admissible boundary condition that is obtained by taking $V_0=\partial V$, $\bbE_0=\partial \bbE$, with $\psi = i$ for some $i\in[q]$ (i.e., the monochromatic spin boundary condition) and $\phi =1$; see Figure~\ref{fig7.1}(c).
In this case,
$\r^{\psi,\phi}$ is the random-cluster measure on the cube 
$R =  \{1,\dots,\ell-1\}^d \subset V$
with wired boundary condition.
On the other hand, the marginal on the spins 
is the Potts measure on 
$
R
$
with the ``all $i$'' monochromatic boundary condition. 

Another relevant random-cluster boundary condition is the one obtained by adding to the random-cluster space the edges ``sticking in'' from $\partial V$. Namely, let $\bbE_1\subset \partial \bbE$ be the set of edges with exactly one endpoint in $\partial V$, and take the monochromatic boundary condition $\psi = i$ and the wired edge boundary condition on $\bbE_0=\partial \bbE\setminus \bbE_1$. 
The marginal on edges is the random-cluster distribution measure on $(V,\bbE\setminus\bbE_0)$ with wired boundary condition on $\partial V$, while
the spin marginal is the Potts measure on 
$R$ with  the ``all $i$'' boundary condition on $\partial V$; see Figure~\ref{fig7.1}(d).

Reasoning in this way
one can obtain, as the edge marginal of the joint measure with an admissible boundary condition, any random-cluster measure with a boundary condition where the vertices in the boundary are either free or wired into a {\em single component}, simply by fixing monochromatic spins on that component and fixing an edge configuration realizing the wiring of that component. 

\bigskip\noindent\textbf{Planar duality.} \ 
A useful tool in two dimensions is planar duality. 
Let $G_\dual = (V_\dual,\EdgeSet_\dual)$
denote the planar dual of $G = (V,\EdgeSet)$, where $V = \{0,\dots,\ell\} \times \{0,\dots,\ell\}$ is a square region of $\bbZ^2$. That is, $V_\dual$ corresponds to the set of faces of $V$, and for each $e \in \EdgeSet$, there is a dual edge $e_\dual \in \EdgeSet_\dual$ connecting the two faces bordering $e$. 
The random-cluster distribution \eqref{eq:rc} satisfies $\rho_{G,p,q}(A) = \rho_{G_\dual,p_\dual,q}(A_\dual)$, where $A_\dual$ is the dual configuration to $A \subseteq \EdgeSet$; i.e., $e_\dual \in A_\dual$ iff $e\not\in A$), and
$$
p_\dual =  \frac{q(1-p)}{q(1-p)+p}.
$$
The self-dual point (i.e, the value of $p$ such that $p=p_d$) corresponds to the critical threshold $p_c(q) = 1-\exp(-\beta_c(q))$.

Since $V_\dual$ is not a subset of $\bbZ^2$, it is convenient to consider the graph $\hat G_\dual = (\hat V_\dual,\hat \EdgeSet_\dual)$ 
with $\hat V_\dual = \{-1,\dots,\ell\} \times \{-1,\dots,\ell\}+(\frac 12,\frac 12)$
and identify all boundary vertices of $\hat V_\dual$ with the vertex of $G_\dual$ corresponding to its external face. Then, $\rho_{G,p,q}^1(A) = \rho_{\hat G_\dual,p_\dual,q}^0(A_\dual)$ and $\rho_{G,p,q}^0(A) = \rho_{\hat G_\dual,p_\dual,q}^1(A_\dual)$, where the $0$ and $1$ superscripts denote the free and and wired boundary conditions respectively (see Section 6.1 in~\cite{Grimmett} for a detailed discussion).

Observe that both random cluster measures $\rho_{\hat G_\dual,p_\dual,q}^1$ and $\rho_{\hat G_\dual,p_\dual,q}^0$ on $\hat G_\dual$ can be obtained as marginals of the joint measure in a square region of $\bbZ^2$ with a monochromatic admissible boundary condition as described above.

%

\subsection{SW dynamics for the random-cluster model}

Our first result concerns the SW dynamics for the random-cluster model.
In this variant of the SW dynamics, given an edge configuration $A$, we assign spins to the connected components of $A$ uniformly at random to obtain a joint configuration, and then update the edge configuration by percolating on the monochromatic edges with probability $p$. The transition matrix $\PSWE$ of this chain satisfies
$$
\PSWE(A,B) = \sum_{\sigma: A \subseteq M(\sigma)} \nu(\sigma \tc A) \nu(B \tc \sigma);
$$
$\PSWE$ is reversible with respect to $\rho$; see, e.g.,~\cite{ES,Ullrich1}.
The following lemma follows from Theorem~\ref{th:swtheorem}.

\begin{lemma}\label{cor:SWmarg}
	Let $\nu := \nu^{\psi,\phi}$ be the joint distribution with an admissible boundary condition $(\psi,\phi)$.
	If $q$ and $\beta = \ln(\frac{1}{1-p})$ are such that SSM holds, 
	then
	the SW dynamics on random-cluster configurations 
	with boundary conditions inherited from $(\psi,\phi)$ satisfies the discrete time entropy decay with rate $\d$, and its mixing time is bounded by $O(\log n)$.
\end{lemma}

\begin{proof}
	If $f$ depends only on the edge configuration, then 
	\begin{align}
	\label{eq:rc-sw}
	\PSWE f (A)=  \nu[\nu(f\tc\si)\tc A] = TKf (\si,A).
	\end{align}
	Here and below,  with slight abuse of notation, if a function $f$ on the joint space depends only on the edge configuration, we again write $f$ for the corresponding (projection) function on edges. Therefore, we have $\ent_\rho(\PSWE f) = \ent_\nu( TK f)$. 
	More precisely, for any $f\geq 0$ depending only on the edge  configuration, and such that  
	$\rho[f]=\nu[f]=1$, one has
	\begin{align*}
	\ent_\rho(\PSWE f) = \rho[(\PSWE f) \log (\PSWE f) ]= \nu [(TKf)\log( TK f)] = \ent_\nu(KTf).  
	\end{align*}
	Theorem \ref{th:swtheorem} says that, for any function $f$ in the joint space, one has
	$$
	\ent_\nu[KTf]\leq (1-\d)\ent_\nu(f).
	$$  
	In particular, for our $f$,
	\begin{align*}
	\ent_\rho (\PSWE f) \leq (1-\d) \ent_\nu(f) = (1-\d)\ent_\rho( f).  
	\end{align*}
	This is the desired discrete time entropy decay for $\PSWE$ in the edge space. 
\end{proof}

\begin{remark}
	The same argument in the previous proof applies to the spin dynamics. 
	In particular, if $g$ is a function depending only on the spin configuration, then
	$\PSW g (\s)= KTg (\si,A)$. Repeating the previous steps with $KT$ in place of $TK$ one has discrete time entropy decay with rate $\d$ for the SW dynamics on spin configurations. This provides an alternative view of the proof of Theorem~\ref{thm:intro:sw} as a corollary of Theorem~\ref{th:swtheorem} for the joint space.
\end{remark}

In $\bbZ^2$, we can take advatange of self-duality of the random-cluster model to obtain bounds for the SW dynamics in the low temperature regime.  

\begin{theorem}
	\label{thm:sw-rc-lt}
	In an $n$-vertex square region of $\bbZ^2$ with free or wired boundary conditions, for all integer $q \ge 2$ and all $p > p_c(q)$, 
	there exists a constant $\d>0$ such that
	for all functions $f:\{0,1\}^\EdgeSet \mapsto \bbR_+$
	\begin{align*}
	\ent_\rho( \PSWE f) \leq \left(1-\frac{\d}{n}\right) \ent_\rho(f).
	\end{align*}
	In particular, the mixing time of the SW dynamics on random-cluster configurations satisfies $\tmix(\PSWE)=O(n\log n)$.
\end{theorem}

Let $G=(V,\EdgeSet)$ where $V$ is $n$-vertex square region of $\bbZ^2$.
Let $\r := \r_{G,p,q}^\theta$ where $\theta \in \{0,1\}$ and
let $\PHB$ be the transition matrix of the heat-bath Glauber dynamics on $G$.
This is the standard Markov chain that, from a random-cluster configuration $A_t \subseteq \EdgeSet$, transitions to a new configuration $A_{t+1}\subseteq \EdgeSet$ as follows:
\begin{enumerate}
	\item choose an edge $e\in \EdgeSet$ uniformly at random;
	\item let $A_{t+1} = A_t \cup \{e\}$ with probability
	$$
	\frac{\rho(A_t \cup \{e\})}{\rho(A_t \cup \{e\})+\rho(A_t \setminus \{e\})} = \left\{\begin{array}{ll}
	\frac{p}{q(1-p)+p} & \mbox{if $e$ is a ``cut-edge'' in $(V,A_t)$;} \\
	p & \mbox{otherwise;}
	\end{array}\right.
	$$
	\item otherwise, let $A_{t+1} = A_t \setminus \{e\}$.
\end{enumerate}
We say $e$ is a {\it cut-edge} in $(V,A_t)$ if the number of connected components in $A_t\cup \{e\}$ and $A_t \setminus \{e\}$ differ.
$\PHB$ is (by design) reversible with respect to $\rho$.
It is also straightforward to check that with the free (resp., wired) boundary condition and parameters $p$ and $q$, for any pair of configurations $A$ and $B$, we have
$\PHB(A,B) = \PHB^\dual(A_\dual,B_\dual)$, where $\PHB^\dual$ denotes the transition matrix of the heat-bath chain on $\hat G_\dual$ with wired (resp., free) boundary condition and paramaters $p_\dual$ and $q$. 

Theorem~\ref{thm:sw-rc-lt} follows from the following two results.

\begin{lemma}
	\label{lemma:df1}
	There exists a constant $c > 0$ such that,
	for every function $f: \{0,1\}^{\EdgeSet} \mapsto \bbR$,
	$$
	\mathcal D_{\PSWE} (f,f) \ge c\cdot \mathcal D_{\PHB} (f,f).
	$$
\end{lemma}

\begin{lemma}
	\label{lemma:df2}
	For all integer $q \ge 2$ and all $p > p_c(q)$, there exists a constant $\delta > 0$ such that,
	for every function $f: \{0,1\}^{\EdgeSet} \mapsto \bbR_+$,
	$$
	\mathcal D_{\PHB} (\sqrt{f},\sqrt{f}) \ge \frac{\delta}{n} \cdot \ent_\rho(f).
	$$
\end{lemma}

\begin{proof}[Proof of Theorem~\ref{thm:sw-rc-lt}]
	Lemmas~\ref{lemma:df1} and~\ref{lemma:df2} imply
	\begin{equation}
	\label{dirian16}
	\mathcal D_{\PSWE} (\sqrt{f},\sqrt{f}) \ge \frac{c \delta}{n} \cdot \ent_\rho(f).
	\end{equation}
	In words, this says that the SW dynamics on random-cluster configurations when $p > p_c(q)$ satisfies a standard log-Sobolev inequality with constant $\frac{c \delta}{n}$. An inequality of Miclo relating the standard log-Sobolev inequality and discrete time entropy decay (see Proposition 6 in~\cite{Miclo}) shows that \eqref{dirian16} implies the entropy decay bound 
	$$
	\ent_\r(\PSWE f) \leq \left(1-\frac{\d c}{n}\right)\ent_\r(f),
	$$
	and the mixing time bound follows from Lemma~\ref{lem:mix} and Remark~\ref{rmk:entropy-mixing} since $\PSWE = \PSWE^*$.
\end{proof}	

It remains to prove Lemmas \ref{lemma:df1} and~\ref{lemma:df2}.
We note 
that a version of the comparison inequality in Lemma~\ref{lemma:df1}
was proved in~\cite{Ullrich1} (see Theorem~4.8 there),  but it is stated for the spectral gap under the free boundary condition. 

In both of these proofs, we consider the single-bond variant of the Glauber dynamics.
In one step of this chain every connected component is assigned a spin from $[q]$ uniformly at random; a random edge $e$ is then chosen and if the endpoints of $e$ are monochromatic, then the edge is added to the configuration with probability $p$
and deleted otherwise. The state of $e$ does not change if its endpoints are bi-chromatic.
Note that this chain is the projection onto edges of the local dynamics on the joint space, see~\eqref{locdyn}; in particular, the update at the edge $e$ corresponds to $W_e$.
Let $\PSB$ denote the transition matrix of the single bond dynamics, which is reversible with respect to $\r$.
The Dirichlet form associated to this chain satisfies
	\begin{align}\label{dirian4}
	\cD_{\PSB}(f,f)
	&= \langle (I-\PSB)f,f\rangle_\r 
	=\r\left[((I-\PSB)f) \cdot f\right] 
	=\frac1{|\bbE|}\sum_{e\in\bbE}\nu\left[\var_\nu(f\tc\si,A_{\bbE\setminus{e}})\right]
	\end{align}
since 
$$
\PSB f (A)=\frac1{|\bbE|}\sum_{e\in\bbE} \nu\left[\nu[f\tc\si,A_{\bbE\setminus{e}}]\tc A\right],
$$	
where with a slight abuse of notation (here and below) we use $f$ also for the ``lift'' of $f$ to the joint space.

We note that for some constants $c_i=c_i(q,p)>0$, $i=1,2$,
\begin{align*}
c_1\PSB(A,B)\leq \PHB(A,B)\leq c_2\PSB(A,B) 
\end{align*}
for all random-cluster configurations $A,B$. Therefore the same bounds apply to the Dirichlet forms: 
\begin{align}\label{dirian6}
c_1\cD_{\PSB}( f, f)\leq \cD_{\PHB}( f, f)\leq c_2\cD_{\PSB}( f, f), 
\end{align}
for any function $f: \{0,1\}^\EdgeSet \mapsto \bbR$.

\begin{proof}[Proof of Lemma~\ref{lemma:df1}]
The Dirichlet form associated with $\PSWE$ is given by 
\begin{align*}
\cD_{\PSWE}(f,g)
&=\langle (I-\PSWE)f,g\rangle_\r = \r\left[((I-\PSWE)f) \cdot g\right],
\end{align*}
and since $\PSWE f(A) = \nu[\nu[f\tc\s]\tc A]$, we obtain
\begin{align*}
\cD_{\PSWE}(f,f) & = \nu\left[(f-\nu[\nu[f\tc\s]\tc A]) \cdot f\right] =\nu\left[(f-\nu[f\tc\si]) \cdot f\right]
=\nu\left[\var_\nu(f\tc\si)\right].
\end{align*}
Then, for any function $f\geq 0$, 
\begin{align*}
\cD_{\PSWE}(\sqrt f,\sqrt f)&=
\nu\left[\var_\nu(\sqrt f\tc\si)\right]\\
&\geq \frac1{|\bbE|}\sum_{e\in\bbE}\nu\left[\var_\nu(\sqrt f\tc\si,A_{\bbE\setminus{e}})\right]\\
&= \cD_{\PSB}(\sqrt f,\sqrt f), 
\end{align*}
where we have used~\eqref{dirian4} and the fact that, for any $e\in\bbE$,
$$
\nu\left[\var_\nu(\sqrt f\tc\si)\right]\geq \nu\left[\var_\nu(\sqrt f\tc\si,A_{\bbE\setminus{e}})\right]
$$
by monotonicity of the variance functional.
The result then follows from~\eqref{dirian6}.
\end{proof}

\begin{proof}[Proof of Lemma~\ref{lemma:df2}]
	By duality (see discussion at the beginning of the section), we have
	\begin{align}\label{dirian10}
	\cD_{\PHB}(\sqrt f,\sqrt f)=\cD_{\PHB^\dual}(\sqrt {f_\dual},\sqrt {f_\dual}),
	\end{align}
	where $f_\dual$ is the function such that $f_\dual(A_\dual) = f(A)$
	and $\PHB^\dual$ is the transition matrix corresponding to the dual of $\r$. 
	
	Thus, if $\cD_{\PHB}$ is at low temperature ($p > p_c(q)$), then $\cD_{\PHB^\dual}$ is at high temperature  ($p < p_c(q)$). Moreover, from~\eqref{dirian4} and~\eqref{dirian6},
	\begin{align*}
	\cD_{\PHB^\dual}(\sqrt  f,\sqrt  f)&\geq c_1\cD_{\PSB^\dual}(\sqrt  f,\sqrt f)
	=\frac{c_1}{|\bbE|}\sum_{e\in\bbE}\nu_\dual\left[\var_{\nu_\dual}(\sqrt  f\tc\si,A_{\bbE\setminus{e}})\right],
	\end{align*}
	where $\nu_\dual$ is the dual joint measure. 
	Specifically, if $\r_\dual$ is the dual measure of $\rho$
	(and the stationary distribution of $\PHB^\dual$), 
	$\nu_\dual$ is a joint measure whose edge marginal is $\rho_d$.
	Observe that since $\rho$ is a random-cluster distribution on the square region $V = \{0,\dots,\ell\} \times \{0,\dots,\ell\}$ with free (or wired) boundary condition, $\rho_d$ is a distribution over $\hat V_\dual = \{-1,\dots,\ell\} \times \{-1,\dots,\ell\}+(\frac 12,\frac 12)$ with wired (or free) boundary condition. As discussed earlier, in either case there exists a joint measure with an admissible boundary condition whose edge marginal is $\rho_d$.  
	
	Observe also that, as before, with a slight abuse of notation,
	we also use $f$ for the ``lift'' of $f$ to the joint space.
	Now, as in \eqref{rough1} we know that for some constant $C=C(p,q)$, for all $e\in\bbE$ and for all $f\geq 0$,
	\begin{align*}
	\var_{\nu_\dual}(\sqrt  f\tc\si,A_{\bbE\setminus{e}})\geq C^{-1}  \ent_{\nu_\dual}( f\tc\si,A_{\bbE\setminus{e}}).
	\end{align*}
	Therefore,
	\begin{align}\label{dirian13}
	\cD_{\PHB^\dual}(\sqrt  f,\sqrt  f)&\geq \frac{c_1C^{-1}}{|\bbE|}\sum_{e\in\bbE}{\nu_\dual}\left[\ent_{\nu_\dual}( f\tc\si,A_{\bbE\setminus{e}})\right].
	\end{align}
	Since for $p < p_c(q)$ and $q \ge 2$ the SSM property holds, we can use \eqref{swent106} to obtain
	\begin{align}\label{dirian14}
	\sum_{e\in\bbE}{\nu_\dual}\left[\ent_{\nu_\dual}( f\tc\si,A_{\bbE\setminus{e}})\right]\geq \d_1 \ent_{\nu_\dual} (f).
	\end{align}
	Indeed, if $f$ is a function of edges only then the first term on the right hand side of \eqref{swent106} is zero. Moreover
	for such an $f$ we have $\ent_{\nu_\dual} (f)=\ent_{\r_\dual} (f)$. Summarizing, we have proved, for all $f\geq 0$,
	\begin{align}\label{dirian15}
	\cD_{\PHB^\dual}(\sqrt  {f_\dual},\sqrt  {f_\dual})&\geq \frac{\d_2}n\,\ent_{\r_\dual} (f_\dual),
	\end{align}
	for a suitable constant $\d_2  > 0$. The result follows from \eqref{dirian10} and the fact that 
	$\ent_{\r_\dual} (f_\dual) = \ent_{\r} (f)$.
\end{proof}	

\begin{remark}
	We remark that \eqref{dirian15} says that the heat-bath Glauber dynamics for the random-cluster model in square regions of $\bbZ^2$ with free or wired boundary conditions satisfies the standard log-Sobolev inequality with constant $\d/n$ for some $\d=\d(p,q)$ for all $p\neq p_c(q)$. This bound is optimal up to a multiplicative constant, as can be seen by choosing an appropriate test function.
\end{remark}

\subsection{Decay for spins from decay for edges and vice versa}
\label{subsec:comparison}

We will use Theorem~\ref{thm:sw-rc-lt} to deduce our low temperature results for the SW dynamics on spin configurations. We do so using the following entropy contraction ``transfer'' result between the spin and edge variants of the SW dynamics.
A similar comparison result for the spectral gap was provided by Ullrich~\cite{Ullrich1}.

\begin{lemma}\label{pro:SWs}
	Suppose we know that the SW dynamics on edges with invariant measure $\r$, corresponding to an $n$-vertex square region $V$ with some boundary condition, has entropy decay with rate $\d$. Then the SW dynamics on spins on $V$, with any boundary condition inherited from a joint measure $\nu$ whose marginal on edges equals $\r$, satisfies the same entropy decay (asymptotically) and has the same mixing time bound $\tmix = O(\d^{-1}\log n)$.  The same applies with the roles of spins and edges reversed.
\end{lemma}
\begin{proof}
	The assumption on $\r$ says that 
	\begin{align}\label{jj8}
	\ent_\r (\wt P_{SW}g) \leq (1-\d) \ent_\r( g),  
	\end{align}
	for any function $g=g(A)$, $A \subset \EdgeSet$. Recalling \eqref{eq:rc-sw} we see that
	\eqref{jj8} can be rewritten as 
	\begin{align}\label{jj9}
	\ent_\nu( TK g) \leq (1-\d) \ent_\nu(g),  
	\end{align}
	for any $g=g(A)$ and any joint measure $\nu$ such that the marginal on edges equals $\r$.
	Now, let $f=f(\si)$ be any function depending only on the spin configuration. Since $g=Tf$ depends only on the edge configuration, we have 
	\begin{align}\label{jj10}
	\ent_\nu (TKT f) \leq (1-\d) \ent_\nu (Tf).  
	\end{align}
	If we apply \eqref{jj10} with $f$ replaced by  $(KT)^{\ell-1} f$, then
	\begin{align}\label{jj11}
	\ent_\nu (T(KT)^{\ell} f) \leq (1-\d) \ent_\nu (T(KT)^{\ell-1} f),  
	\end{align}
	for any $\ell\in\bbN$. Iterating this inequality we find, for any $\ell\in\bbN$,
	\begin{align}\label{jj12}
	\ent_\nu( T(KT)^{\ell} f) \leq (1-\d)^\ell \ent_\nu (T f). 
	\end{align}
	Recalling that  $P^\ell_{SW}f =(KT)^\ell f$, from \eqref{jj12} we get
	\begin{align*}
	\ent_\mu(P^\ell_{SW}f) &= \ent_\nu( (KT)^{\ell} f)\\&
	= \ent_\nu (KT(KT)^{\ell-1} f)\\&
	\leq \ent_\nu(T(KT)^{\ell-1} f)\\&
	\leq (1-\d)^{\ell-1}\ent_\nu(Tf) \\&\leq (1-\d)^{\ell-1}\ent_\nu( f)
	= (1-\d)^{\ell-1}\ent_\mu(f),
	\end{align*}
	where the first inequality follows from~\eqref{conv7}.
	This shows that the discrete time entropy decay for SW on spins is asymptotically the same as the one assumed for SW on edges, and Lemma \ref{lem:mix} allows us to conclude the desired mixing time bound. 
	The same argument (with $KT$ replaced by $TK$) shows that if we assume an entropy decay for spins then we obtain (asymptotically) the same entropy decay for edges, and therefore the same mixing time bound.  
\end{proof}

We can now provide the proof of Theorem~\ref{cor:sw-lt:intro} from the introduction.

\begin{proof}[Proof of Theorem~\ref{cor:sw-lt:intro}]
	From the discussion at the beginning of Section~\ref{sec:rc}, note that there is an admissible boundary condition in the joint space for which the edge marginal is the random-cluster measure on a square region of $\bbZ^2$ with a wired boundary condition, and the spin marginal is the monochromatic boundary condition. 
	The result then follows from Theorem~\ref{thm:sw-rc-lt} and Lemma~\ref{pro:SWs}.
\end{proof}

\bibliographystyle{alpha}


 \end{document}